\documentclass[11pt,reqno,oneside]{amsart}
\usepackage[a4paper, left=2cm, right=2cm, top=2.5cm, bottom=2cm]{geometry}
\usepackage{amsmath}
\usepackage{amsthm}
\usepackage{amsfonts}
\usepackage{amssymb}
\usepackage{xcolor}
\usepackage{enumitem}
\usepackage[mathcal]{euscript}
\setlist[enumerate]{label={\rm(\roman*)}, leftmargin=2em}

\usepackage{mathtools}
\usepackage{hyperref}
\usepackage[sort&compress]{natbib}
\hypersetup{
	colorlinks=true,
	linkcolor=teal,
	citecolor=violet,
	urlcolor=violet}
\usepackage{graphicx}

\usepackage[utf8]{inputenc}
\usepackage[T1]{fontenc}

\theoremstyle{plain}
\newtheorem{theorem}{Theorem}[section]
\newtheorem{lemma}[theorem]{Lemma}
\newtheorem{corollary}[theorem]{Corollary}
\newtheorem{proposition}[theorem]{Proposition}
\newtheorem{mainthm}{Theorem}

\theoremstyle{definition}
\newtheorem{remark}[theorem]{Remark}
\newtheorem{definition}[theorem]{Definition}
\newtheorem{example}[theorem]{Example}

\numberwithin{theorem}{section}
\numberwithin{equation}{section}

\DeclarePairedDelimiter\abs{\lvert}{\rvert}
\DeclarePairedDelimiter\nrm{\lVert}{\rVert}
\DeclarePairedDelimiter\set{\lbrace}{\rbrace}
\DeclarePairedDelimiter\brk{(}{)}

\let\brc\set

\renewcommand\d{\mathrm{d}}
\newcommand\dd{\,\d}
\newcommand{\hra}{\hookrightarrow}
\makeatletter
\newcommand{\hrastar}{%
  \mathrel{\mathop{\hra}\limits^{
    \vbox to 0ex{\kern-2\ex@
    \hbox{$\scriptstyle*$}\vss}}}}
\makeatother

\newcommand\measurable{\mathcal{M}}
\newcommand\measurablep{\mathcal{M}_+}
\newcommand\aefinite{\mathcal{M}_0}

\newcommand{\N}{\mathbb{N}}
\newcommand{\R}{\mathbb{R}}

\newcommand{\RR}{\mathcal R}
\renewcommand{\SS}{\mathcal S}
\newcommand{\MM}{\mathcal M}
\newcommand{\FF}{\mathcal F}
\newcommand{\LL}{\mathcal L}

\newcommand{\DD}{\mathcal D}

\newcommand{\vp}{\varphi}
\newcommand{\vpx}{\varphi_X}
\newcommand{\vpxp}{\varphi_{X'}}
\newcommand{\vpy}{\varphi_Y}

\newcommand{\nm}{{\frac{n}{m}}}
\newcommand{\mn}{{\frac{m}{n}}}
\newcommand{\ip}{{\frac{1}{p}}}
\newcommand{\iq}{{\frac{1}{q}}}
\newcommand{\ib}[1][f]{%
\if#1f\frac{1}{\beta}
\else\if#1i{1/\beta}
\fi\fi
}
\newcommand{\ai}{{\alpha-1}}

\newcommand{\tb}{t^{\beta}}

\newcommand{\RM}{(\RR,\mu)}
\newcommand{\SN}{(\SS,\nu)}
\newcommand{\MRM}{\MM\RM}
\renewcommand{\MR}{\mu(\RR)}

\newcommand\lb{{t_0}}
\newcommand\ub{{t_\infty}}
\newcommand\LLM{\{L,\Lambda,M\}}

\usepackage{xspace}
\makeatletter
\DeclareRobustCommand\onedot{\futurelet\@let@token\@onedot}
\def\@onedot{\ifx\@let@token.\else.\null\fi\xspace}
\def\eg{e.g\onedot} 
\def\ie{i.e\onedot} 
\def\cf{cf\onedot} 
\def\ae{a.e\onedot} 
\def\ri{r.i\onedot} 
\def\etc{etc\onedot}
\makeatother

\makeatletter
\def\paragraph{\bigskip\@startsection{paragraph}{4}%
  \z@\z@{-\fontdimen2\font}%
  {\normalfont\bfseries}}
\makeatother

\input{widebar}

\title{Optimality problems in Orlicz spaces}

\author{V\'\i t Musil\textsuperscript{1}}
\address{\textsuperscript{1}%
Institute for Theoretical Computer Science -- Faculty of Informatics, Masaryk University,
Botanick\'a 554/68a, Ponava, Brno
Czech Republic}

\author{Lubo\v s Pick\textsuperscript{2}}
\address{\textsuperscript{2}%
Department of Mathematical Analysis,
Faculty of Mathematics and Physics,
Charles University,
So\-ko\-lo\-vsk\'a~83,
186~75 Praha~8,
Czech Republic}

\author{Jakub Tak\'a\v c\textsuperscript{2}}

\email[Musil, corresponding author]{musil@fi.muni.cz}

\urladdr[Musil]{0000-0001-6083-227X}
\urladdr[Pick]{0000-0002-3584-1454}
\urladdr[Tak\'a\v c]{0000-0003-2158-7456}

\begin{document}

\begin{abstract}
In mathematical modelling, the data and solutions are represented as measurable functions and their quality is oftentimes captured by the membership to a certain function space.
One of the core questions for an analysis of a model is the mutual relationship between the data and solution quality.
The optimality of the obtained results deserves a special focus.
It requires a careful choice of families of function spaces balancing between their expressivity, \ie the ability to capture fine properties of the model, and their accessibility, \ie its technical difficulty for practical use.
This paper presents a unified and general approach to optimality problems in Orlicz spaces.
Orlicz spaces are parametrized by a single convex function and neatly balance the expressivity and accessibility.
We prove a general principle that yields an easily verifiable necessary and sufficient condition for the existence or the non-existence of an optimal Orlicz space in various tasks.
We demonstrate its use in specific problems, including the continuity of Sobolev embeddings and boundedness of integral operators such as the Hardy--Littlewood maximal operator and the Laplace transform.
\end{abstract}

\maketitle

\section{Introduction}

On a high level, many practical tasks can be modelled as an assignment that maps input data to a solution.
For example, an evolution of a system in physics or economics can be captured by a partial differential equation (PDE).
Measured or observed data then produce the prediction of the future behaviour of the system as the solution of this PDE.
One of the key questions is what can be \emph{a priori} said about the \emph{solution's quality} based on the data quality.
Oftentimes, the data and solutions are represented by measurable functions, and their ``quality'' is expressed by certain conditions they obey (measurements they would pass).
The collection of all functions satisfying given conditions is called a \emph{function space}.
Therefore, results of this kind may be abstracted as
\begin{equation} \label{E:TXY}
	T\colon X\to Y,
\end{equation}
where $T$ is the operator mapping the data $f$ to the solution $Tf$, and $X$ and $Y$ are function spaces that describe the quality of the data or the solution, respectively.
Moreover, if the ``quality conditions'' define a structure like a (quasi-)normed space, the boundedness of $T$ carries further information about the transfer of the ``quality'' from data to the solution.

Some data and solution conditions may be more strict than others, which translates to an inclusion (or embedding) of the corresponding function spaces.
Therefore, among all results of the form~\eqref{E:TXY}, some may be better than others.
The ultimate aim is to find the \emph{optimal results}.
For instance, given $X$, one seeks the smallest $Y$ satisfying relation~\eqref{E:TXY}.
This translates to the goal of finding the \emph{best} guaranteed quality of prediction ($Y$), assuming data satisfies given assumptions ($X$).
Analogously, we may ask about the \emph{weakest} assumptions on the data so that any solution achieves the prescribed quality.
This we can model as finding the largest $X$ in relation~\eqref{E:TXY} when $Y$ is fixed.

Optimality problems have always been of interest, and, in the last two decades, they have seen a real boom.
A limited assortment of references is given below in Section~\ref{SS:related-work}.
However, the exposed problem is not yet well defined unless we specify the class of competing function spaces representing the quality conditions.
Here, we need to compromise between the class's \emph{expressivity}, \ie its overall richness and versatility, and its \emph{accessibility}, \ie its complexity and technical difficulty.
The very existence of an optimal function space in a given class is often a highly nontrivial problem.
Considering for example the scale of \emph{Lebesgue $L^p$ spaces}, a simple, well-known one-parameter family of spaces (high accessibility), the problem is that the class is oftentimes too poor to offer optimal results (low expressivity).
In contrast, experience shows that among the so-called \emph{rearrangement-invariant spaces} (\emph{\ri~spaces} for short; also known as \emph{symmetric} or \emph{K\"othe} \emph{spaces}) an optimal space more or less always exists (high expressivity).
This is counterbalanced by the fact that the optimal \ri~space is described in an implicit way and, as such, is of little practical value unless some simplification is available (low accessibility).

This paper presents a \emph{unified and general approach} to optimality problems in \emph{Orlicz spaces}.
Orlicz spaces are parametrized by a single convex function (called \emph{Young function}).
They are fairly easy to understand while offering a lot of flexibility.
The pool of Orlicz spaces is ideally balanced; it is rich enough to cover cases in which Lebesgue spaces have no say, but it is still reasonably accessible for practical use.\looseness=-1

\subsection{General Result}

To illustrate the situation on the solution side, consider a fixed operator~$T$ and assume that $X$ is given.
The optimality problem in Orlicz spaces now reduces to the question of whether there exists the smallest Orlicz space, say $L^B$, containing the image $Y=T[X]$.
For the purpose of this presentation, let us assume that $Y$ is an \ri space\footnote{Later, $Y$ will be the smallest \ri space containing $T[X]$ once its existence is granted.}.

Every \ri~space $Y$ has its \emph{fundamental function}, $\varphi_Y\colon[0,\infty)\to[0,\infty]$, describing the behaviour of norm on indicator functions, \ie $\varphi_Y(t)=\nrm{\chi_E}_Y$ for any set $E$ with measure $t$.
For example, Lebesgue spaces have $\smash{\varphi_{L^p}(t)=t^{1/p}}$ for $p\in[1,\infty)$.
We define the \emph{fundamental level} as the collection of all \ri~spaces sharing the same fundamental function.
There are various function spaces whose fundamental functions coincide (are equivalent in some sense).
For instance, all two-parameter Lorentz spaces $\smash{L^{p,q}}$ obey $\smash{\varphi_{L^{p,q}}=t^{1/p}}$ with any $q\in[1,\infty]$.
However, there is a \emph{unique Orlicz space on each fundamental level}, and an explicit formula connecting the fundamental function and the Young function of the Orlicz space is well known.
Interestingly, this very simple fact is a crucial feature of our theory.

For an~\ri~space $Y$, we call its \emph{fundamental Orlicz space}, denoted by $L(Y)$, the unique Orlicz space with the same fundamental function as $Y$.
One of our most interesting discoveries is that the existence or non-existence of an optimal Orlicz space is fully governed by the \emph{positioning} of $Y$ and its fundamental Orlicz space $L(Y)$.

\begin{mainthm}[the principal alternative]\label{T:intro-principal-alternative-for-spaces}
\phantom{.}
\begin{enumerate}
\item\label{en:PA-spaces-target} Let $Y$ be an \ri space and $L(Y)$ its fundamental Orlicz space.
Then either $Y\subset L(Y)$ and $L(Y)$ is the smallest Orlicz space containing $Y$,
or $Y\not\subset L(Y)$ and no smallest Orlicz space containing $Y$ exists.
\item\label{en:PA-spaces-domain} Let $X$ be an \ri space and $L(X)$ its fundamental Orlicz space.
Then either $L(X)\subset X$ and $L(X)$ is the largest Orlicz space contained in $X$,
or $L(X)\not\subset X$ and no largest Orlicz space contained in $X$ exists.
\end{enumerate}
\end{mainthm}

Note that the fact that the smallest Orlicz space containing $Y$ exists when $Y\subset L(Y)$ is easy.
The difficult part of the theorem is that, remarkably, the converse is also true.
Analogous considerations also hold for the domain side~\ref{en:PA-spaces-domain}.

\subsection{Applications to Sobolev Embeddings}

We illustrate the possible use of Theorem~\ref{T:intro-principal-alternative-for-spaces} on Sobolev embeddings.
In its simplest form, first-order Sobolev embedding corresponds to the solution operator $T$ mapping a given data function $f$ to the weakly differentiable solution $u$ of the equation
\begin{equation} \label{E:nabla-u-f}
	\nabla u = f
		\quad\text{on $\Omega$}
		\quad\text{and}\quad
		u = 0
		\quad\text{on $\R^n\setminus\Omega$},
\end{equation}
where $\Omega$ is a fixed open connected set in $\R^n$.
The natural domain $\DD$ of $T$ consists of all locally integrable functions $f$ on $\Omega$ for which a weakly differentiable solution $u$ of equation~$\eqref{E:nabla-u-f}$ uniquely exists.
The transfer of regularity for equation~\eqref{E:nabla-u-f} then reads as
\begin{equation} \label{E:nabla-u-f-T}
	T\colon \DD\cap X(\Omega) \to Y(\Omega),
\end{equation}
where $X$ and $Y$ are some function spaces.
Here, relation~\eqref{E:nabla-u-f-T} tells us what can be generally said about a function $u$ when some information about its gradient $\nabla u$ is available.
This particular scenario is typically formulated as the embedding
\begin{equation} \label{E:Sob-V1-0}
	V^1_0 X(\Omega) \hra Y(\Omega),
\end{equation}
where $V^1_0 X(\Omega)$ is so-called first-order Sobolev space defined as a collection of weakly-differentiable functions $u$ satisfying $\abs{\nabla u}\in X(\Omega)$ and $u=0$ outside $\Omega$.
Assuming that $X$ is a Banach space, we equip the Sobolev space $V^1_0 X(\Omega)$ with the norm $\smash{\nrm{u}_{V^1_0 X}=\nrm{\nabla u}_X}$.
Formulation~\eqref{E:Sob-V1-0} bypasses the question of existence and uniqueness of solutions of equation~\eqref{E:nabla-u-f} and focuses on the problem of regularity transfer from $\nabla u$ to $u$.
Therefore, we stick to this notation in the rest of this section.

We give an overview of first-order Sobolev embeddings within the classes of Lebesgue, Orlicz and rearrangement-invariant spaces ranging from long-standing classical results to new ones.
The principle alternative allows us to easily reconstruct all known solutions to optimality problems in Orlicz spaces without methods tailored to each special case.
We will focus on illustrating the above-mentioned expressivity and accessibility of the discussed classes and we shall exhibit and stress the power of the underlying principal alternative.
The details and exhaustive treatment of general Sobolev embeddings of arbitrary order and without zero boundary condition are collected in Section~\ref{S:Principal-alternative-for-sobolev-embeddings}.

\paragraph{Optimal Targets}

Let us fix the domain space $X=L^p$ with $p\in[1,n]$.
In the subcritical case $p<n$, the classical Sobolev inequality asserts that
\begin{equation} \label{E:classical-sobolev-embedding}
    V^1_0 L^p(\Omega) \hra L^{\frac{np}{n-p}}(\Omega),
\end{equation}
\cf \eg~\citep{Ada:75,Ada:03,Maz:11}.
Simple counterexamples show that one cannot replace $\cramped{L^\frac{np}{n-p}(\Omega)}$ in~\eqref{E:classical-sobolev-embedding} by any essentially smaller Lebesgue space without losing the validity of the embedding.
In other words, $\smash{L^{\frac{np}{n-p}}(\Omega)}$ is the otimal Lebesgue target space in embedding~\eqref{E:classical-sobolev-embedding}.
On the other hand, if $p=n$, there is no smallest Lebesgue space $L^q(\Omega)$ such that
\begin{equation}\label{E:classical-sobolev-embedding-limiting}
    V^1_0 L^{n}(\Omega) \hra L^{q}(\Omega).
\end{equation}
Indeed, embedding~\eqref{E:classical-sobolev-embedding-limiting} holds for every $q\in[1,\infty)$, but not for $q=\infty$ itself, \cf \eg~\citep{Ada:75,Ada:03,Maz:11}.
We see that Lebesgue spaces are not rich enough to capture the optimality or, in other words, their expressivity is insufficient.

Let us look at these embeddings when we enrich the pool of competing spaces with the Orlicz spaces.
In the case $p<n$, it turns out that the space $\smash{\cramped{L^{\frac{np}{n-p}}(\Omega)}}$ in~\eqref{E:classical-sobolev-embedding} is smallest also among all Orlicz spaces,
see \eg~\citep{Cia:96}.
More importantly, the class of Orlicz spaces has sufficient expressivity to solve the optimality problem also in the limiting case $p=n$.
Indeed, the Trudinger inequality yields
\begin{equation} \label{E:classical-sobolev-embedding-limiting-Orlicz-concrete}
    V^1_0 L^n(\Omega) \hra \exp L^{\frac{n}{n-1}}(\Omega),
\end{equation}
where the target is an Orlicz space whose Young function is equivalent to $\smash{\cramped{e^{t^{\frac{n}{n-1}}}}}$ near infinity, \cf \eg~\citep{Tru:67,Poh:65,Yud:61}.
The space $\exp L^{\frac{n}{n-1}}(\Omega)$ is then the smallest among all Orlicz spaces as shown by~\citet{Hem:70} or~\citet{Cia:96}.

Finally, we consider these embeddings using the even much larger and more abstract family of so-called rearrangement-invariant spaces,
and we demonstrate how to obtain optimality in Orlicz spaces using the principal alternative.
The sub-limiting embedding~\eqref{E:classical-sobolev-embedding} may be improved to the optimal one
\begin{equation} \label{E:classical-sobolev-embedding-lorentz}
    V^1_0 L^p(\Omega) \hra L^{\frac{np}{n-p},\,p}(\Omega),
\end{equation}
where, on the target side, we are dealing with the two-parameter Lorentz space, \cf~\citep[Example~7(1)]{Edm:00}.
For just the embedding without optimality, see \eg~\citep{One:63,Pee:66,Hun:66}.
Now, since $\smash{\cramped{L^{\frac{np}{n-p},\,p}(\Omega)\hra L^{\frac{np}{n-p}}(\Omega)}}$, the latter is the Orlicz space, and they are both on the same fundamental level, principal alternative implies the Orlicz optimality of the target in~\eqref{E:classical-sobolev-embedding}.

Analogously, in the limiting case $p=n$, embedding~\eqref{E:classical-sobolev-embedding-limiting-Orlicz-concrete} can be further improved when allowing \ri target spaces by
\begin{equation}\label{E:bw}
    V^1_0 L^{n}(\Omega) \hra L^{\infty,\,n;\,-1}(\Omega),
\end{equation}
in which the right hand side stands for a particular instance of a Lorentz-Zygmund space~\citep{Maz:11,Han:79,Bre:80}, see also~\citep{Bru:79} and~\citep{Cwi:00}.
The fundamental Orlicz space of $\smash{\cramped{L^{\infty,n;-1}(\Omega)}}$ coincides with the exponential space $\smash{\cramped{\exp L^{\frac{n}{n-1}}(\Omega)}}$.
Since $\smash{\cramped{L^{\infty,n;-1}(\Omega)\hra \exp L^{\frac{n}{n-m}}(\Omega)}}$, the optimality of the Orlicz target in~\eqref{E:classical-sobolev-embedding-limiting-Orlicz-concrete} follows by our principal alternative.

From a general point of view, \citet{Cia:96} proved that to every Orlicz space $\smash{L^A(\Omega)}$, there always exits the optimal Orlicz space $\smash{L^B(\Omega)}$ satisfying
\begin{equation} \label{E:sob-orlicz}
    V^1_0 L^A (\Omega) \hra L^{B}(\Omega).
\end{equation}
This suggests that Orlicz spaces have high enough expressivity for this task.

\paragraph{Optimal Domains}

Let us turn our attention to optimal domains.
In subcritical embedding~\eqref{E:classical-sobolev-embedding}, the domain space $L^p(\Omega)$ is optimal within Lebesgue and also Orlicz spaces, see \eg~\citep[Example~5.1]{Mus:16}.
In our scheme, we may argue that the optimal \ri space is $\smash{\cramped{L^{p,\frac{np}{n-p}}(\Omega)}}$, see \eg~\citep{Edm:00}.
Its fundamental Orlicz space is $L^p(\Omega)$ and since it is smaller than $\smash{\cramped{L^{p,\frac{np}{n-p}}(\Omega)}}$, it is the optimal one by the principal alternative.

In the limiting cases, the behaviour changes drastically. The optimal \ri domain space in
\begin{equation}\label{E:embedding-into-ell-infty}
    V^1_0 X(\Omega) \hra L^{\infty}(\Omega)
\end{equation}
happens to be $X=L^{n,\,1}$, see \eg~\citep{Ste:81} for the embedding and~\citep{Edm:00} for the optimality.
It follows immediately from the principal alternative that no optimal domain Orlicz space exists since the fundamental Orlicz space of $L^{n,\,1}$ is $L^n$, which is strictly larger.
Note that the non-existence of the optimal Orlicz space, in this case, was shown before by \citet[Theorem~6.4]{CiP:98}.
However, compared to the versatile nature of the principal alternative, the original proof requires a non-trivial construction designed specifically for dealing with the target space~$L^\infty$.

Similarly, in embedding~\eqref{E:classical-sobolev-embedding-limiting-Orlicz-concrete}, no optimal domain Orlicz space exists, although $L^{n}$ is the optimal domain Lebesgue space.
This was first proved by another explicit construction, which was developed by Kerman and Pick, and appears in~\citep{Pic:98} and \citep{Pic:02}.
Once again, an alternative proof is available via comparing the optimal \ri domain $X$, see \eg~\cite[Example~7(3)]{Edm:00}, with its fundamental Orlicz space $L(X)$ and using the principal alternative.

We see that, compared to the target side, the situation on the domain side is more delicate, and Orlicz spaces do not suffice for capturing the optimality.
In general, to any given weak Orlicz (also called Marcinkiewicz) space, a characterisation of the existence of the optimal domain Orlicz space is given by~\citet{Mus:16}, which generalizes the aforementioned constructions of~\citet{CiP:98} and \citet{Pic:98}.
Relying on this result, necessary and sufficient conditions for the existence of $L^A$ in embedding~\eqref{E:sob-orlicz} for any Orlicz target space $L^B(\Omega)$ were given by~\citet{Cia:19}.
In our Theorem~\ref{T:sobolev-to-orlicz}, we present a much easier and more comprehensible approach.
We show that the fundamental Orlicz of an optimal \ri domain depends only on the fundamental level of the target; see Corollary~\ref{C:fundamental-domain-equality}.
The proof then pops out immediately from the principal alternative and known reduction principles.\looseness=-1

Returning back to our examples, let us examine the last case of limiting embedding~\eqref{E:bw}.
Specifically, it has not been known whether there exists a largest Orlicz space $L^A(\Omega)$ such that
\begin{equation} \label{E:bw-orlicz}
    V^1_0 L^A(\Omega) \hra L^{\infty,\,n;\,-1}(\Omega).
\end{equation}
As the space $\smash{\cramped{L^{\infty,n;-1}(\Omega)}}$ is neither an Orlicz nor a weak Orlicz space, none of the existing results in the literature give us a conclusive answer.
However, the principal alternative allows us to solve this long-standing open problem.
We provide a proof of a far stronger Theorem~\ref{T:nonexistence-optimal-orlicz-on-level} asserting that if no optimal domain Orlicz space exists for some target weak Orlicz space $Y$, then there is no optimal domain Orlicz space also for any target space on the fundamental level of $Y$.
Finally, the spaces $\smash{\cramped{L^{\infty,n;-1}(\Omega)}}$ and $\smash{\cramped{\exp L^{\frac{n}{n-1}}(\Omega)}}$ are on the same fundamental level and the latter is a weak Orlicz space for which no optimal Orlicz domain space in~\eqref{E:classical-sobolev-embedding-limiting-Orlicz-concrete} exists, as we know due to the earlier constructions mentioned above.
Consequently, no optimal Orlicz space $L^A$ in embedding~\eqref{E:bw-orlicz} exists either.

\subsection{Applications to Classical Operators}

Apart from Sobolev embeddings, the question of the existence of an optimal Orlicz space was open for several important operators such as the Hardy--Littlewood maximal operator, the Laplace transform, the Riesz potential, the Hilbert transform or the Fourier transform.
The optimal \ri~spaces are known in most of these cases, see~\citep{Edm:20,Bur:17}, except perhaps for the Fourier transform.
A notable exception from this list seems to be the fractional maximal operator for which the optimal Orlicz spaces have been fully characterized by~\citet{Mus:19}.
In Section~\ref{S:principal-alternative-for-operators}, we shall provide a complete characterisation of the existence of optimal Orlicz target space for the Hardy-Littlewood Maximal operator with a fixed Orlicz domain.
In cases when the optimal target exists, we provide an explicit description of the relevant Young functions.
We shall also present some progress towards a satisfactory characterisation of the existence of optimal target and domain Orlicz spaces for the Laplace transform.

\subsection{Related work}\label{SS:related-work}

There is a vast amount of literature devoted to the optimality of function spaces in various tasks and applications and a comprehensive survey is out of the scope of this paper.
Instead, we mention a few examples that are close to our contribution.

The classical paper of~\citet{Cal:66} treats optimal pairs of spaces from the interpolation point of view.
The celebrated paper of \citet{Tru:67} (the result can also be traced in the works of \citet{Yud:61} and \citet{Poh:65}) solved the problem of the nonexistence of optimal Lebesgue target space in limiting Sobolev embedding by finding an appropriate target space in the broader family of Orlicz spaces.
In a follow-up paper, \citet*{Hem:70} showed that, among Orlicz spaces, the target space is optimal.
Trudinger's result has been refined in many ways, and simple proofs were found, probably the most notable one by \citet{Str:71}.

An important breakthrough in the optimality direction was achieved by \citet{Mos:70} who quantified Trudinger's exponential inequality, found the optimal constant and proved that it is attained. An analogous result for higher-order embeddings was obtained by~\citet{Ada:88}.
In connection with the study of quantum fields and hypercontractivity semigroups, extensions of the classical Sobolev embeddings to the setting when the underlying measure space is infinite-dimensional have been investigated by \citet{Gro:75} who showed that in the case of Gaussian measure, a dimension-free embedding can be obtained involving a logarithmic amendment on the side of the target space.
This paper demonstrated once again the power of Orlicz spaces and opened new paths of vast research, part of which produced simple proofs and optimality results of all kinds, see \eg~the work of \citet{Tal:76} or \citet*{Pel:93}, references therein and follow-up papers.

A very interesting approach to optimal pairs of norms was taken by \citet{Ker:79}.
In the 1990s, the importance of the study of optimality in the specific context of Orlicz spaces started resurfacing mainly thanks to the efforts of Cianchi, see \eg his classical work~\citep{Cia:96}.
By the outbreak of the 21st century, optimal function spaces started appearing in broader contexts, see for example the series of papers on optimal growth envelopes by Caetano, Haroske, Triebel and others, see \eg~\citep{Cae:01}, its references and follow-up papers.

The boom of optimality results in the last two decades can be illustrated by many works, let us briefly mention some of them.
Curbera and Ricker produced a series of papers on optimality in the context of rearrangement-invariant spaces, see \eg~\citep{Cur:02}.
\citet{Cia:04} studied optimality of function spaces in Orlicz--Sobolev embeddings.
\citet*{Cia:12} sought fine improvements of Sobolev embeddings using optimal remainder gradient norms.
\citet*{Bre:21} studied sharp embeddings for Sobolev spaces involving symmetric gradients.
Optimal Calderón spaces for generalized Bessel potentials were nailed down by \citet*{Bak:21}.
Optimal embeddings of Calder\'on spaces in H\"{o}lder--Zygmund spaces were established by \citet*{Bas:14}.
Optimal Sobolev-type inequalities in the very important class of Lorentz spaces were studied by \citet*{Cas:13}.
Optimal rearrangement-invariant Sobolev embeddings in mixed norm spaces were established by \citet{Cla:16}.
Applications of optimality problems to sharp regularity estimates for PDEs were studied by \citet*{daS:18}, see also the work of \citet*{Sal:15}.
Certain optimal limiting embeddings for reduced Sobolev spaces were obtained by \citet{Fon:14}.
Optimal mapping properties of the fractional integral operators, Fourier integral operators and the $k$-plane transforms on rearrangement-invariant quasi-Banach function spaces were studied by \citet{HoK:21}.
Optimal Sobolev--Lorentz embeddings involving mean oscillation were treated by \citet{Iok:14}.
The optimal behaviour of the Fourier transform and the convolution operator on compact groups was considered by \citet*{Kum:20}.
\citet{Mih:21} was able to extend the existing work on the optimality of Sobolev embeddings on bounded domains to the case of entire Euclidean space, augmenting thereby existing results by \citet{Vyb:07} and \citet*{Alb:18}.
The optimal performance of Fourier multipliers on Lebesgue spaces was studied by \citet*{Moc:10}.
Optimal local embeddings of Besov spaces were under certain restrictions established by~\citet*{Nev:20}.
The recent work of \citet{Pes:22} brings a new point of view on classical Wiener amalgam spaces by combining it with the idea of symmetrization.

This is just a very short exhibit of the important and dynamically developing field, and the list is by no means complete.
Further results, namely those which intimately concern our work, will be mentioned later, mainly in the introduction to Section~\ref{S:Principal-alternative-for-sobolev-embeddings}.

\subsection{How to Read the Paper}

In Section~\ref{S:unions}, we study unions of Orlicz spaces contained in a given space and intersections of Orlicz spaces containing a given space.
In particular, we prove that in the case of strong and weak Orlicz spaces (also known as Lorentz and Marcinkiewicz endpoint spaces) these sets coincide.
This is in fact a certain generalization of similar a result of~\citet{MZP}, but it is used later as a crucial ingredient in the proof of the main results.
Some serious technical obstacles are overcome in the same section, in particular, we establish a new Young-type inequality in Proposition~\ref{P:OL-inequality}.
We focus on working with as general structures as possible, in particular, we consider quasi-convex functions in place of Young ones.

Section~\ref{S:principal-alternative} is devoted to the proof of Theorem~\ref{T:intro-principal-alternative-for-spaces}.
Section~\ref{S:Principal-alternative-for-sobolev-embeddings} has a specification of the principal alternative to Sobolev-type embeddings and several of its consequences.
Section~\ref{S:principal-alternative-for-operators} is then devoted to applications of the principal alternative to the Hardy--Littlewood maximal operator and to the Laplace transform.

The paper is finished with an Appendix dedicated to a~detailed study of the so-called \emph{absolutely continuous} or \emph{almost compact} embeddings between function spaces.
Such relations have been proved very useful for studying the compactness of operators and embeddings, see \eg~\citep{Sla:12,Fer:10}.
Roughly speaking, compactness follows from the combination of almost compactness with compactness in measure.
This material was obtained as a certain side effect of the mainstream of our work, but it is definitely of independent interest.
We push further the knowledge on the relations of endpoint spaces to unions and intersections of Orlicz spaces mentioned above.
We introduce the notion of (\emph{uniform}) \emph{sub-diagonality} of an \ri~space, meaning that it coincides with the union of Orlicz spaces embedded into it (almost compactly).
We provide a sufficient condition on a Young function $A$ such that for the Lorentz endpoint space $X$ satisfying $L(X)=L^A$, the space $X$ is uniformly sub-diagonal.
We then establish certain lifting principle which enables us to transfer (uniform) sub-diagonality along scales of function spaces.
The combination of these results allows us to provide conditions on (uniform) sub-diagonality on various scales of function spaces.
In particular, this idea allows us to obtain previously unknown useful relations between classical Lorentz spaces (in particular $L^{p,q}$ spaces) and Orlicz spaces.
The reason for having collected this material in an Appendix is that it leads to a different direction of research than the main results of the paper.

We do not have a preliminary section, all the important notions are introduced at the beginning of the sections in which they appear for the first time.

\section{Unions and Intersections of Orlicz Spaces}\label{S:unions}

Our main aim in this section is to establish key relations between function spaces, with particular emphasis on the set-wise identities between (Lorentz/Marcinkiewicz) endpoint spaces and unions/intersections of Orlicz spaces.
Throughout the section, we work with spaces of real measurable functions defined on the real half-line $(0,\infty)$ equipped with the Lebesgue measure.
This does not present any significant restriction of the theory as the results would remain valid for function spaces containing real functions defined on any $\sigma$-finite measure space.
The details of such extension are discussed later in Section~\ref{S:principal-alternative}.

We will denote by $\measurable$ the (non-linear) space of all measurable
functions $f\colon (0,\infty)\to[-\infty,\infty]$. By $\measurablep$ and
$\aefinite$ we denote the collections of all non-negative functions from
$\measurable$ and of all almost everywhere finite functions from
$\measurable$, respectively. If $E\subset (0,\infty)$ is measurable, we denote
by $\abs{E}$ its Lebesgue measure.

We first recall technical operations on functions of one variable, such as generalized inverses and rearrangements.
We then fix the notation of quasi-convex and Young functions which we use to define the Orlicz spaces.
We shall also define ``strong'' and ``weak'' variants of Orlicz spaces and recall basic relations between them.
For details, we refer to \citep{Rao:91} and \citep{BS}.

\paragraph{Inverses and the correlative function}

Let $F\colon[0,\infty]\to[0,\infty]$ be a non-decreasing function.
The \emph{right-continuous inverse} of $F$ is defined by
\begin{equation} \label{E:def-rc}
	F^{-1}(t)
		= \sup\set{\tau\ge 0: F(\tau)\le t}
	\quad\text{for $t\in[0,\infty]$}.
\end{equation}
We have the trivial inequality
\begin{equation} \label{E:comp-rc}
	t\le F^{-1}\bigl(F(t)\bigr)
	\quad\text{for $t\in[0,\infty]$}.
\end{equation}
For a non-decreasing function $G\colon[0,\infty]\to[0,\infty]$,
its \emph{left-continuous inverse} is given as
\begin{equation} \label{E:def-lc}
	G^{-1}(t)
		= \inf\set{\tau\ge 0: G(\tau)\ge t}
	\quad\text{for $t\in[0,\infty]$}
\end{equation}
and obeys
\begin{equation} \label{E:comp-lc}
	G^{-1}\bigl(G(t)\bigr) \le t
	\quad\text{for $t\in[0,\infty]$}.
\end{equation}
Note that we do not introduce any notation to distinguish the
left or right continuity of an inverse; instead we always mention it in the
text explicitly.

Let $F\colon[0,\infty]\to[0,\infty]$.
We define its~\emph{correlative function} by
\begin{equation} \label{E:def-correlative-function}
	F_\#(t) = \frac{1}{F(\tfrac{1}{t})}
		\quad\text{for $t\in[0,\infty]$.}
\end{equation}
The symbol $F_\#^{-1}$ abbreviates the correlative function of a right-continuous
inverse, \ie $F_\#^{-1}=(F^{-1})_\#$.  Note that if $F$ is strictly
increasing and surjective, then $(F^{-1})_\#=(F_{\#})^{-1}$ where both the
inverses are taken in the classical sense.

\paragraph{The distribution function and the non-increasing rearrangement}

Given $f\in\measurable$, its \emph{distribution function}, $f_*$, is defined as
\begin{equation} \label{E:def-distribution}
	f_*(\lambda) = \abs{\set{\tau\in(0,\infty): \abs{f(\tau)}>\lambda}}
		\quad\text{for $\lambda\in[0,\infty]$}.
\end{equation}
The \emph{non-increasing rearrangement}, $f^*$, of $f$ is given by
\begin{equation} \label{E:def-rearrangement}
	f^*(t) = \inf\set{\lambda\ge 0: f_*(\lambda) \le t}
		\quad\text{for $t\in[0,\infty]$}.
\end{equation}
Both $f_*$ and $f^*$ are non-increasing, right-continuous
and $(f^*)_*=f_*$ on $[0,\infty]$.

For $f\in\measurable$, the \emph{elementary maximal function} $f^{**}$
is defined as
\begin{equation} \label{E:def-maximal}
	f^{**}(t) = \frac{1}{t} \int_0^t f^{*}(\tau)\dd\tau
		\quad\text{for $t\in(0,\infty)$.}
\end{equation}
The function $f^{**}$ is non-increasing on $(0,\infty)$ and for every
$f,g\in\measurable$, one has
\begin{equation} \label{E:maximal-triangle}
	(\abs{f}+\abs{g})^{**} \le f^{**} + g^{**}
		\quad\text{on $(0,\infty)$.}
\end{equation}

\paragraph{Relations between scalar quantities}

We write $L\lesssim R$ if $L\le cR$ holds for some constant $c>0$ independent of quantities involved in $L$ and $R$.
If both $L\lesssim R$ and $R\lesssim L$ hold, we write $L\approx R$ and say that $L$ and $R$ are \emph{equivalent}.
When $L$ and $R$ are functions of a real variable $t$,
we write $L\prec R$, if there is some $K>0$ such that $L(t)\le R(Kt)$.
If simultaneously $L\prec R$ and $R\prec L$, we write $L\sim R$.

Furthermore, we use the terminology \emph{near infinity}, \emph{near zero} and \emph{globally} to indicate that any of the properties holds for $t\ge t_0$, resp.\ $0\le t\le t_0$, with some $t_0\in(0,\infty)$, resp.\ for all $t\ge 0$.
If not specified, we assume the global variant.

\paragraph{Quasi-convex and Young functions}

A function $A\colon [0,\infty] \to [0,\infty]$ is called \emph{quasi-convex} if it is left-continuous, not identically zero or infinity on $(0,\infty]$, $A(0)=0$, and $t\mapsto A(t)/t$ is non-decreasing on $(0,\infty)$.
If, moreover, $A$ is convex on $[0,\infty]$, then it is called a \emph{Young function}.
A non-decreasing function $A\colon[0,\infty]\to[0,\infty]$ is quasi-convex if and only if $A_\#$ is quasi-convex.

We note that $A$ is allowed to attain infinite value, which is the only case in which the assumption of the left continuity is relevant.
Every Young function $A$ can be uniquely expressed as
\begin{equation}\label{E:Young-as-integral}
	A(t) = \int_0^t a(\tau) \dd\tau
		\quad \text{for $t\in[0,\infty]$},
\end{equation}
where $a$ is a non-decreasing right-continuous function that will be called the \emph{right-continuous derivative} of $A$.
We also adopt the notation $A'=a$.
On the interval where $A$ is finite, one has $A'=a$ \ae in the classical sense.
It holds that
\begin{equation} \label{E:Young-trivial}
	A(t)\le ta(t) \le A(2t)
		\quad\text{for $t\in[0,\infty]$}.
\end{equation}
It follows from \eqref{E:Young-as-integral} and the monotonicity of $a$ that every Young function is quasi-convex. Conversely, if $B$ is quasi-convex, then the function $A$, given by
\begin{equation}\label{E:Young-from-quasiconvex}
	A(t) = \int_0^t \frac{B(\tau)}{\tau} \dd\tau
		\quad\text{for $t\in[0,\infty]$,}
\end{equation}
is a Young function that satisfies
\begin{equation}\label{E:Young-basic-equivalence}
	A(t)\le B(t) \le A(2t)
		\quad\text{for $t\in[0,\infty]$.}
\end{equation}
A quasi-convex function $A$ is said to satisfy \emph{the $\Delta_2$ condition} if $A(2t)\le CA(t)$ for some $C>0$ and all~$t\ge 0$.

The \emph{complementary function} $\widetilde A$ of a quasi-convex
function $A$ is given by
\begin{equation} \label{E:conjugate}
	\widetilde{A}(t) = \sup\set{\tau t - A(\tau): \tau>0}
		\quad\text{for $t\in[0,\infty]$.}
\end{equation}
An immediate consequence of the definition is the so-called \emph{Young's inequality}
\begin{equation}\label{E:Young-inequality}
	\tau t\le A(\tau) + \widetilde{A}(t)
		\quad\text{for $\tau,t\in[0,\infty]$.}
\end{equation}
If $A$ is a Young function, then so is $\widetilde A$.
If $A$ moreover enjoys representation~\eqref{E:Young-as-integral}, then
\begin{equation} \label{E:Young-conjugate-as-integral}
	\widetilde A(t) = \int_0^t a^{-1}(\tau) \dd\tau
		\quad \text{for $t\in[0,\infty]$},
\end{equation}
where $a^{-1}$ is the right-continuous inverse of $a$.

If $A$ is a quasi-convex function and $A^{-1}$ is its right-continuous inverse, then
\begin{equation}\label{E:upper-bound-right-continuous-inverse}
    A(A^{-1}(t)) \le t \quad \text{for every $t\in[0,\infty]$.}
\end{equation}
Furthermore, Young function $A$ and its complementary function $\widetilde A$ obey
\begin{equation}\label{E:upper-bound-for-product}
    t\le A^{-1}(t)\widetilde A^{-1}(t) \le 2t
		\quad\text{for every $t\in(0,\infty)$,}
\end{equation}
where both inverses are right-continuous.
For quasi-convex functions $A$ and $B$, it holds that
\begin{equation} \label{E:domination}
	B\prec A
	\quad\text{if and only if}\quad
	\widetilde A\prec\widetilde B
	\quad\text{if and only if}\quad
	A_\#\prec B_\#.
\end{equation}

\paragraph{Orlicz spaces}

For a quasi-convex function $A$, the \emph{Orlicz semi-modular}
$\varrho_A\colon\measurable\to[0,\infty]$ is defined by
\begin{equation*}
	\varrho_A(f) = \int_0^\infty A(\abs{f})
	\quad\text{for $f\in\measurable$.}
\end{equation*}
The collection $L^A=\set{f\in\measurable:\varrho_A(f/\lambda)\le 1\;\text{for
some $\lambda>0$}}$ is called the \emph{Orlicz space}, and the functional
$\nrm{\cdot}_{L^A}\colon\measurable\to[0,\infty]$, given by
\begin{equation}\label{E:luxemburg-norm}
	\nrm{f}_{L^A} = \inf\set{\lambda>0: \varrho_A(f/\lambda)\le 1},
\end{equation}
is called the \emph{Luxemburg norm}.
It is well known (\cf \eg \citep{Rao:91}) that if $A$ is a Young function, then the functional $\nrm{\cdot}_{L^A}$ is a norm on $L^A$, one has $L^A=\set{f\in\measurable: \nrm{f}_{L^A}<\infty}$, and, $L^A$ endowed with the Luxemburg norm is a Banach space.
If $B$ is only quasi-convex, then $\varrho_B$ and $\nrm{\cdot}_{L^B}$ are well defined, but $\nrm{\cdot}_{L^B}$ is not necessarily a norm.
However, if $A$ is the Young function corresponding to $B$ in the sense of~\eqref{E:Young-from-quasiconvex}, then~\eqref{E:Young-basic-equivalence} immediately yields
\begin{equation} \label{E:modular-sandwich}
	\varrho_A(f) \le \varrho_B(f) \le \varrho_A(2f)
		\quad\text{for $f\in\measurable$,}
\end{equation}
which in turn leads to
\begin{equation} \label{E:norm-sandwich}
	\nrm{f}_{L^A}
		\le \nrm{f}_{L^B}
		\le 2\nrm{f}_{L^A}
		\quad\text{for $f\in\measurable$,}
\end{equation}
hence $L^B$ is equivalently normable with equivalence constants 1 and 2.

For a quasi-convex $A$ and $f\in\measurable$, it holds that
\begin{equation} \label{E:Orlicz-balls}
	\varrho_A(f) \le 1
		\quad\text{if and only if}\quad
	\nrm{f}_{L^A} \le 1
\end{equation}
and
\begin{equation} \label{E:Orlicz-ball}
	\varrho_A(f/\nrm{f}_{L^A}) \le 1.
\end{equation}

On taking $A(t)=t^p$, $t\in[0,\infty]$, for some $p\in[1,\infty)$ one obtains $L^A=L^p$, recovering the classical Lebesgue space.
If one defines $A=0$ on $[0,1]$ and $A=\infty$ on $(1,\infty]$, then $L^A=L^\infty$.
Semi-modulars are further discussed in Appendix~\ref{S:unioins-of-orlicz-spaces}.

\paragraph{Strong and weak Orlicz spaces}

Let $A$ be a quasi-convex function. We define the functional
\begin{equation} \label{E:LambdaA-def}
	\nrm{f}_{\Lambda^A}
		= \int_0^\infty A^{-1}_\#(f_*)
	\quad\text{for $f\in\measurable$}.
\end{equation}
The \emph{strong Orlicz space}, also called \emph{Lorentz space}, $\Lambda^A$ is now defined as
\begin{equation}\label{E:MA-def}
	\Lambda^A = \set{f\in\measurable: \nrm{f}_{\Lambda_A}<\infty}.
\end{equation}
The space $\Lambda^A$ is linear.
If $A^{-1}_\#$ is concave, then $\nrm{\cdot}_{\Lambda^A}$ is a norm and $\Lambda^A$ is a Banach space.
Even if $A^{-1}_\#$ is not concave, there exists a concave function $\varphi$ such that $A^{-1}_\#\le\varphi\le2A^{-1}_\#$ on $(0,\infty)$ and $\Lambda^A$ can be equivalently renormed to become a Banach space, see Lemma~\ref{L:LambdaE-G}.

For a quasi-convex function $A$, we define the functional
\begin{equation*}
	\nrm{f}_{M^A} = \sup_{t\in(0,\infty)} A^{-1}_\#(t) f^{**}(t)
		\quad\text{for $f\in\measurable$,}
\end{equation*}
and the \emph{weak Orlicz space} or \emph{Marcinkiewicz space} $M^A$ by
\begin{equation*}
	M^A = \set{ f\in\measurable: \nrm{f}_{M^A}<\infty }.
\end{equation*}
It follows from inequality~\eqref{E:maximal-triangle} that
$\nrm{\cdot}_{M^A}$ is a norm and $M^A$ is a Banach space.

For the purposes of this section, we will denote by $\LLM$ the set of all classical, strong and weak Orlicz spaces generated by quasi-convex functions.

\paragraph{Embeddings}

For Banach spaces $X$ and $Y$, we write $X\hra Y$ if there exists $C>0$ such that $\nrm{f}_{Y}\le C\nrm{f}_X$ for every $f\in X$.
We recall that for $X,Y\in\LLM$ it holds that $X\subset Y$ if and only if $X\hra Y$.

Let $A$ be a quasi-convex function. Then we have the fundamental embeddings
\begin{equation} \label{E:sandwich}
	\Lambda^A \hra L^A \hra M^A
\end{equation}
which explain the meaning of ``weak'' and ``strong'' Orlicz spaces.
Moreover embeddings~\eqref{E:sandwich} are tight in the sense that the norms are indistinguishable on characteristic functions.
That is, for every $E\subset(0,\infty)$ measurable, one has the \emph{fundamental relation}
\begin{equation} \label{E:fundamental}
	\nrm{\chi_E}_{\Lambda^A}
		= \nrm{\chi_E}_{L^A}
		= \nrm{\chi_E}_{M^A}
		= A^{-1}_\#(\abs{E}).
\end{equation}

The embeddings of two (strong/classical/weak) Orlicz spaces are governed by relations of their quasi-convex functions.
Namely, for quasi-convex functions $A$ and $B$, we have that
\begin{equation} \label{E:embeddings}
	\Lambda^A \hra \Lambda^B
	\quad\text{iff}\quad
	L^A\hra L^B
	\quad\text{iff}\quad
	M^A \hra M^B
	\quad\text{iff}\quad
	B\prec A.
\end{equation}

\paragraph{Associate spaces}

For $X\in \set{L,\Lambda,M}$, we define
\begin{equation} \label{E:associate}
	\nrm{f}_{X'} = \sup\set*{ \int_0^\infty fg: \nrm{g}_X\le 1}
		\quad\text{for $f\in\measurable$,}
\end{equation}
and $X'=\set{f\in\measurable\colon \nrm{f}_{X'}<\infty}$. The space $X'$ is called the \emph{associate space} of $X$.
For a quasi-convex function $A$, it holds that
\begin{equation} \label{E:duality}
	\bigl(L^A\bigr)'=L^{\widetilde{A}}
		\quad\text{and}\quad
	\bigl(\Lambda^A\bigr)'=M^{\widetilde{A}}
		\quad\text{and}\quad
	\bigl(M^A\bigr)'=\Lambda^{\widetilde{A}}.
\end{equation}
Here the identity of function spaces means that they coincide in the set-theoretical sense and, moreover, their norms are equivalent up to absolute constants.
Using \eqref{E:embeddings}, \eqref{E:duality} and \eqref{E:domination}, it follows that for $X,Y \in\LLM$, one has $X\hra Y$ if and only if $Y'\hra X'$.
Associate spaces are addressed in a broader sense in Section~\ref{S:principal-alternative} below.

\medskip

The goal of this section is to show that an arbitrary strong Orlicz space, \ie a Lorentz endpoint space $\Lambda^E$ coincides with the union of Orlicz spaces embedded into $\Lambda^E$.
This results has been known in a special case when $E$ is a so-called $N$-function satisfying the $\Delta_2$ condition, see~\citep[Section~9]{MZP}.
We shall show that one only needs to assume that $E$ is a quasi-convex function, though the proof becomes more technical.
We shall work in a far wider setting which allows us to also obtain related results for the so-called classical Lorentz spaces.
The basic tool we shall need is the following proposition founded on Young's inequality.

\begin{proposition} \label{P:OL-inequality}
Let $A$ and $G$ be Young functions and $v$ a weight. Let $g=G'$ and
$g^{-1}$ and $G^{-1}$ denote the left-continuous inverses of $g$ and $G$,
respectively.  Then
\begin{equation} \label{E:OL-inequality}
	\int_{0}^\infty G^{-1}\bigl(f_*(t)\bigr)v(t)\dd t
		\le \int_{0}^\infty g^{-1} \left( \frac{v(t)}{\lambda a(t)} \right)v(t) \dd t
			+ \lambda \int_{0}^\infty A(\abs{f})
\end{equation}
for any $\lambda>0$ and any measurable $f$.
\end{proposition}

\begin{proof}
Let $f$ be given and assume that both the integrals on the right-hand side of
\eqref{E:OL-inequality} are finite.
Denoting $\lb=\sup\set{\tau\ge 0: A(\tau)=0}$ and $\ub=\inf\set{\tau\ge 0: A(\tau)=\infty}$,
we infer that
$f_*(t)=0$ for $t>\ub$ and, by Fubini's theorem,
\begin{equation} \label{E:modular-fubini}
	\int_{0}^\infty A(\abs{f})
		= \int_\lb^\ub a(t)f_*(t) \dd t
		< \infty.
\end{equation}
Young's inequality implies that
\begin{equation*}
	G^{-1}\bigl(f_*(t)\bigr)\frac{v(t)}{a(t)}
		\le G\bigl(G^{-1}(f_*(t))\bigr) + \widetilde{G}\left(\frac{v(t)}{a(t)}\right)
		\quad\text{for $t\in(\lb, \ub)$}.
\end{equation*}
Applying the inequality $G(G^{-1}(\tau))\le\tau$ for Young function~$G$, using
the estimate $\widetilde{G}(\tau)\le \tau g^{-1}(\tau)$,
and multiplying both
sides by~$a(t)$ yields
\begin{equation*}
	G^{-1}\bigl(f_*(t)\bigr) v(t)
		\le a(t) f_*(t) + g^{-1}\left(\frac{v(t)}{a(t)}\right) v(t),
\end{equation*}
which, integrating over $(\lb,\ub)$ and employing equality~\eqref{E:modular-fubini},
gives
\begin{equation} \label{E:OL-middle}
	\int_{\lb}^{\ub} G^{-1}\bigl(f_*(t)\bigr)v(t) \dd t
		\le \int_{0}^\infty A(\abs{f})
			+ \int_\lb^\ub g^{-1}\left(\frac{v(t)}{a(t)}\right)v(t)\dd t.
\end{equation}

Assume that $t_0>0$. We have that $a(t)=0$ for $t\in[0,\lb)$
and the convergence of the first integral on the right-hand side of
\eqref{E:OL-inequality} dictates that
\begin{equation} \label{E:g-invers-infty}
	g^{-1}(\infty) = K < \infty.
\end{equation}
and, consequently
\begin{equation} \label{E:ocas-first}
	\int_0^\lb g^{-1}\left(\frac{v(t)}{a(t)}\right)v(t) \dd t
		= \int_{0}^\lb Kv(t)\dd t.
\end{equation}
Since $g$ is non-decreasing, condition~\eqref{E:g-invers-infty}
also implies that $g^{-1}\le K$ on $[0,\infty]$,
hence $g=\infty$ on $[K,\infty)$ and therefore also $G=\infty$ on
$(K,\infty)$. In conclusion, $G^{-1}\le K$ on $[0,\infty]$ and
therefore
\begin{equation} \label{E:OL-first}
	\int_0^{\lb} G^{-1}\bigl(f_*(t)\bigr)v(t) \dd t
		\le \int_0^\lb Kv(t)\dd t
		= \int_0^\lb g^{-1}\left(\frac{v(t)}{a(t)}\right)v(t) \dd t,
\end{equation}
where the equality is due to~\eqref{E:ocas-first}.

If $t_\infty<\infty$, then, again by the assumed convergence of integrals
on the right hand side of inequality~\eqref{E:OL-inequality}, we infer that
\begin{equation*}
	\int_\ub^\infty g^{-1}\left(\frac{v(t)}{a(t)}\right)v(t) \dd t
		= \int_{\ub}^\infty kv(t)\dd t < \infty,
\end{equation*}
where $k=g^{-1}(0)$. Also $G^{-1}(0)=0$ by the left continuity of $G^{-1}$ and,
as $f_*=0$ beyond $\ub$,
\begin{equation} \label{E:OL-last}
	\int_\ub^\infty G^{-1}\bigl(f_*(t)\bigr)v(t) \dd t
		= \int_\ub^\infty kv(t)\dd t
		= \int_\ub^\infty g^{-1}\left(\frac{v(t)}{a(t)}\right)v(t) \dd t.
\end{equation}
Now, estimates~\eqref{E:OL-middle}, \eqref{E:OL-first} and \eqref{E:OL-last}
combined give
\begin{equation*}
	\int_{0}^\infty G^{-1}\bigl(f_*(t)\bigr)v(t)\dd t
		\le \int_{0}^\infty g^{-1} \left( \frac{v(t)}{a(t)} \right)v(t) \dd t
			+ \int_{0}^\infty A(\abs{f}).
\end{equation*}
The desired inequality~\eqref{E:OL-inequality} then follows once we replace
$a$ by $\lambda a$.
\end{proof}

From Proposition~\ref{P:OL-inequality}, we get as a side result a sufficient condition for an Orlicz space to be embedded into a \emph{classical Lorentz space} $\Lambda_w^q$ defined as a collection of all measurable functions $f$ for which the functional
\begin{equation}\label{E:classical-lorentz}
	\nrm{f}_{\Lambda_w^q}
		= \left( \int_0^\infty f^*(t)^q\,w(t)\dd t \right)^\iq
\end{equation}
is finite. Here $q>0$ and $w$ is a locally integrable nonnegative function.

\newcommand{\ep}{{q}}

\begin{theorem}[embedding of Orlicz space into classical Lorentz space] \label{T:embedding-orlicz-into-classical-Lorentz}
Let $A$ be a Young function, $w$ be a non-increasing, locally integrable, non-negative function and $\ep>0$.
Assume that
\begin{equation} \label{E:suff-condition-Orlicz-embedding-into-Lorentz}
	N_\lambda
		= \int_0^{\infty} W\bigl(w^{-1}\bigl(\lambda a(t)t^{1-\ep}\bigr)\bigr)\,t^{\ep-1}\dd t
		< \infty,
\end{equation}
for some $\lambda>0$,
where $W(t)=\int_{0}^t w$ and $w^{-1}$ is the right-continuous inverse of $w$.
Then
\begin{equation} \label{E:Orlicz-Lorentz-embedding}
	\nrm{f}_{\Lambda_w^\ep} \le (\ep N_\lambda+\ep\lambda)^\frac1\ep \nrm{f}_{L^A}
		\quad\text{for every $f\in L^A$}.
\end{equation}
\end{theorem}

\begin{proof}
Let $G=W^{-1}$ denote the left-continuous inverse of $W$. Then $G$ is a Young function
and $g^{-1}$, the left-continuous inverse of $g=G'$, obeys
\begin{equation} \label{E:g-invers-Ww}
	g^{-1}(\tau) = W\bigl(w^{-1}(1/\tau)\bigr)
		\quad\text{for every $\tau\in[0,\infty]$}.
\end{equation}
The definition of classical Lorentz functional together with Fubini's theorem give
\begin{equation*}
	\nrm{h}_{\Lambda_w^\ep}^\ep
		= \ep\int_{0}^\infty G^{-1}\bigl(h_*(t)\bigr)\,t^{\ep-1} \dd t
\end{equation*}
for arbitrary measurable $h$.  Application of Proposition~\ref{P:OL-inequality} on a
function $h=f/\nrm{f}_{L^A}$ with $v(t)=t^{\ep-1}$ yields
\begin{align*}
	\frac{\nrm{f}_{\Lambda_w^\ep}^\ep}{\nrm{f}_{L^A}^\ep}
		= \nrm{h}_{\Lambda_w^\ep}^\ep
		\le \ep\int_{0}^\infty g^{-1}\left( \frac{t^{\ep-1}}{\lambda a(t)} \right)\,t^{\ep-1} \dd t
			+ \ep\lambda \int_{0}^\infty A(\abs{h})
		\le \ep N_\lambda + \ep\lambda,
\end{align*}
where we used equality~\eqref{E:g-invers-Ww} and property~\eqref{E:Orlicz-ball}.
The claim now follows by taking the $\ep$-th root.
\end{proof}

In what follows, we are going to replace the correlative function $E_\#$ of a given quasi-convex function $E$ with a close Young function $G$.
This step is purely technical and deserves some comment.
As far as embeddings are concerned, to prove our main results, we shall be mainly interested in embeddings of the spaces $L^A \hra \Lambda^E$ for a couple of quasi-convex functions $A$ and $E$.
It turns out that it is extremely useful to employ a functional equivalent to $\nrm{\cdot}_{\Lambda^E}$, which employs the function $G^{-1}$, depending only on $E$, instead of $E^{-1}_\#$ but enjoying the property that $G^{-1}\approx E^{-1}_\#$ and $G$ is \emph{Young} (not just quasi-convex).
Roughly speaking, the reason for this step consists in the fact, that for a pair of Young functions $G_1$, $G_2$, one has $G_1\prec G_2$ if and only if $G_1'\prec G_2'$, but this property fails if we only assume $G_1$ and $G_2$ to be quasi-convex.
This may cause serious technical troubles in the case when $E^{-1}_\#$ is \emph{not concave}, and that may happen even if $E$ is Young.

\begin{lemma} \label{L:LambdaE-G}
Let $E$ be a quasi-convex function and denote
\begin{equation} \label{E:t-0-infty-def}
	\tau_0 = \inf\set{\tau\ge 0: E(\tau)>0}
		\quad\text{and}\quad
	\tau_\infty = \sup\set{\tau\ge 0: E(\tau)<\infty}.
\end{equation}
Define
$G\colon[0,\infty]\to[0,\infty]$ by
\begin{equation} \label{E:G-def}
	G(t) = \int_{0}^{t}\frac{E_\#(\tau)}{\tau}\dd\tau
		\quad\text{for $t\in[0,\infty]$}
\end{equation}
and set $t_0=1/\tau_\infty$ and $t_\infty=1/\tau_0$.
Then $G$ is a Young function that is zero on $[0,t_0]$, increasing on $(t_0,t_\infty]$ and infinity on $(t_\infty,\infty]$.
The left-continuous inverse $G^{-1}$ is increasing on $(0,G(t_\infty)]$, constant $G^{-1}=t_\infty$ on $[G(t_\infty),\infty]$, continuous on $(0,\infty]$, continuous at zero if and only if $E$ is finite-valued, and  it is $G^{-1}(t)\to t_0$ as $t\to 0_+$.
Moreover
\begin{equation} \label{E:inverse-sharp-ineq}
	E^{-1}_\#(t) \le G^{-1}(t) \le 2E^{-1}_\#(t)
		\quad\text{for $t\in[0,\infty]$}
\end{equation}
and
\begin{equation} \label{E:LambdaE-G}
	\nrm{f}_{\Lambda^E}
		\le \int_{0}^\infty G^{-1}(f_*)
		\le 2\nrm{f}_{\Lambda^E}
	\quad\text{for $f\in\measurable$}.
\end{equation}
\end{lemma}

\begin{proof}
Quasi-convexity of $E$ implies that $E_\#$ is quasi-convex, whence $E_\#(t)/t$ is non-decreasing and $G$ is a Young function.
Next, $E=0$ on $[0,\tau_0]$ and $E_\#=\infty$ on $[t_\infty,\infty]$, hence $G=\infty$ on $(t_\infty,\infty]$.
Analogously, $E=\infty$ on $(\tau_\infty,\infty]$ and $E_\#=0$ on $[0,t_0)$, thus $G=0$ on $[0,t_0]$.
Function $G$ is increasing and continuous on $(t_0,t_\infty]$ and maps this interval onto $(0,G(t_\infty)]$, therefore $G^{-1}$ on $(0,G(t_\infty)]$ is the classical continuous increasing inverse of $G$.
For $t\ge G(t_\infty)$, we have
	$G^{-1}(t)
		= \inf\set{\tau\ge 0: G(\tau)\ge t}
		= \inf\set{\tau\ge 0: G(\tau)=\infty}
		= t_\infty$.
Altogether, $G^{-1}$ is continuous on $(0,\infty]$.
Finally,
	$\lim_{t\to 0_+} G^{-1}(t)
		= \inf\set{\tau\ge 0: G(\tau)>0}
		= t_0$
and, as by definition $G^{-1}(0)=0$, $G^{-1}$ is continuous at zero if and only if $t_0=0$ which happens if and only if $E$ is finite-valued.

By trivial estimates~\eqref{E:Young-trivial}, it is $G(t)\le E_\#(t)\le G(2t)$, which is equivalent to $G_\#(\tau)\le E(2\tau)\le G_\#(2\tau)$.
Passing to inverses, we obtain that
\begin{equation} \label{E:inverse-sharp-ineq-1}
	E^{-1}_\#(t) \le (G_\#)^{-1}_\#(t) \le 2E^{-1}_\#(t)
		\quad\text{for $t\in[0,\infty]$.}
\end{equation}
Now, by the definition of the right-continuous inverse, we have
\begin{equation*}
	(G_\#)^{-1}\bigl(\tfrac1t\bigr)
		= \inf\set[\big ]{\lambda\ge 0: G_\#(\lambda)> \tfrac1t}
		= \inf\set[\big ]{\tfrac1\sigma: t> G(\sigma)}
\end{equation*}
and therefore
\begin{equation*}
	(G_\#)^{-1}_\#(t)
		= \frac{1}{\inf\set[\big ]{\frac1\sigma: t> G(\sigma)}}
		= \sup \set{\sigma: t> G(\sigma) }
		= G^{-1}(t),
\end{equation*}
where the last equality is the definition of the left-continuous~$G^{-1}$.
Consequently, relation~\eqref{E:inverse-sharp-ineq-1} rewrites as \eqref{E:inverse-sharp-ineq}.
Inequalities~\eqref{E:LambdaE-G} then follow from \eqref{E:inverse-sharp-ineq} and definition~\eqref{E:LambdaA-def}.
\end{proof}

\begin{lemma} \label{L:w}
Let $E$ be a quasi-convex function and $G$, $t_0$, $t_\infty$ be as in Lemma~\ref{L:LambdaE-G}.
Define $w\colon(0,\infty)\to[0,\infty]$ by
\begin{equation} \label{E:w-def}
	w(\tau) = G^{-1}(\tau) E\left( \frac{1}{G^{-1}(\tau)} \right)
	\quad\text{for $\tau\in(0,\infty)$},
\end{equation}
where $G^{-1}$ is the left-continuous inverse of $G$.
Then $w$ is non-increasing, zero on $[G(t_\infty),\infty)$ and
\begin{equation}
\label{E:w-integral}
	G^{-1}(t) = t_0 + \int_{0}^{t} w
		\quad\text{for $t\in(0,\infty]$}.
\end{equation}
Furthermore,
\begin{equation} \label{E:Lambda-norm-weight}
	\int_{0}^{\infty} G^{-1}(f_*)
		= t_0\nrm{f}_\infty
			+ \int_{0}^{\infty} f^*w
	\quad\text{for $f\in\MM$}.
\end{equation}
\end{lemma}

\begin{proof}
Denote $g(t)=E_\#(t)/t$ for $t\in(0,\infty)$.
Then $g$ is nondecreasing on $(0,\infty)$, $G^{-1}$ is positive, finite and increasing on $(0,G(t_\infty))$, hence $w(\tau)=1/g(G^{-1}(\tau))$ is nonincreasing on $(0,G(t_\infty))$.
If $\tau\ge G(t_\infty)$, then $G^{-1}(\tau)=t_\infty$ and $E(1/t_\infty)=0$ and therefore $w(\tau)=0$.
Now, as $G$ is locally absolutely continuous and $G'=g$ \ae on $(t_0,t_\infty)$, we have by the change of variables $\tau=G(r)$ \citep[see \eg][Corollary 20.5]{Hew:75} that
\begin{equation*}
	\int_{0}^t w(\tau)\dd\tau
		= \int_{0}^t \frac{\d\tau}{g\bigl(G^{-1}(\tau)\bigr)}
		= \int_{t_0}^{G^{-1}(t)} \d r
		= G^{-1}(t) - t_0
	\quad\text{for $t\in(0,G(t_\infty))$}.
\end{equation*}
Equality~\eqref{E:w-integral} holds also for $t\ge G(t_\infty)$ as $G^{-1}(t)=t_\infty$ constantly and $w(t)=0$ here.
Finally, identity~\eqref{E:Lambda-norm-weight} follows by \eqref{E:w-integral}, Fubini's theorem and the fact that $\inf\set{t\ge 0: f_*(t)=0}=\nrm{f}_\infty$, as
\begin{equation*}
	\int_{0}^{\infty} G^{-1}\bigl(f_*(t)\bigr)\dd t
		= \int_{0}^{\nrm{f}_\infty}
			\left(
				t_0 + \int_{0}^{f_*(t)} w(\tau)\dd\tau
			\right) \d t
		= t_0\nrm{f}_\infty
			+ \int_{0}^{\infty} f^*(\tau)w(\tau)\dd\tau.
		\qedhere
\end{equation*}
\end{proof}

Relations \eqref{E:LambdaE-G} and \eqref{E:Lambda-norm-weight} are of particular importance as they enable us to somehow separate $f$ from $E$ within the definition of $\nrm{f}_{\Lambda^E}$.
Note that the expression on the right-hand side of \eqref{E:Lambda-norm-weight} resembles the classical definition of the Lorentz spaces, see~\eg~\citep[Chapter~2, Eq.~5.19]{BS}.

We are now ready to provide a sufficient condition for the embedding $L^A\hra \Lambda^E$, which will allow us to show, in the subsequent theorem, that each function in $\Lambda^E$ is contained in some smaller Orlicz space.

\begin{theorem}[Orlicz-Lorentz Embedding]
\label{T:embedding-orlicz-into-endpoint-Lorentz}
Let be $A$ a Young function and $E$ a quasi-convex function.
Let $g^{-1}$ denote the left-continuous inverse of $g(t)=E_\#(t)/t$.
If
\begin{equation} \label{E:union-function-inequality}
	N_\lambda
		= \int_0^\infty g^{-1}\left(\frac{1}{\lambda a(t)}\right) \d t
		< \infty
	\quad\text{for some $\lambda>0$},
\end{equation}
then then $L^A \hra \Lambda^E$ and
\begin{equation} \label{E:LA-LambdaE-emb}
	\nrm{f}_{\Lambda^E}
		\le (N_\lambda + \lambda) \nrm{f}_{L^A}
	\quad\text{for $f\in L^A$}.
\end{equation}
\end{theorem}

\begin{proof}
Let $f\in L^A$ be a non-zero function and set $h=f/\nrm{f}_{L^A}$ and denote $G(t)=\int_0^t g$.
Proposition~\ref{P:OL-inequality} and Lemma~\ref{L:LambdaE-G} combined give
\begin{equation}
	\frac{\nrm{f}_{\Lambda^E}}{\nrm{f}_{L^A}}
		= \nrm{h}_{\Lambda^E}
		\le \int_0^\infty G^{-1}(h_*(s))\dd s
		\le \int_0^\infty g^{-1}\left(\frac{1}{\lambda a(t)}\right)\d t
			+ \lambda\int_0^\infty A(\abs{h})
		\le N_\lambda + \lambda
\end{equation}
since, by~\eqref{E:Orlicz-balls}, $\rho_A(h)\le 1$ as
$\nrm{h}_{L^A}=1$ and inequality~\eqref{E:LA-LambdaE-emb} follows.
\end{proof}

\begin{theorem}\label{T:union-absolutely-continuous}
Let $E$ be a quasi-convex function.
To each nonzero $f\in\Lambda^E$, there exists a Young function $A$ obeying condition~\eqref{E:union-function-inequality} and such that $f\in L^A$.
\end{theorem}

\begin{proof}
Let $f\in\Lambda^E$ be given, denote $\lambda=2\nrm{f}_{\Lambda^E}\ne0$ and set $h=f/\lambda$.
Let $G$ be the Young function associated with $E$ as in \eqref{E:G-def} and $G^{-1}$ the left-continuous inverse of~$G$.
Lemma~\ref{L:LambdaE-G} asserts that
\begin{equation} \label{E:Wh-one}
	\int_{0}^{\infty} G^{-1}(h_*)
		\le 2\nrm{h}_{\Lambda^E}
		= \frac2\lambda \nrm{f}_{\Lambda^E}
		= 1.
\end{equation}
Let $w$ denote the function from \eqref{E:w-def} and set $w(0)=\infty$ and $w(\infty)=\lim_{r\to\infty} w(r)$.
Since $w$ is nonincreasing, we get by \eqref{E:w-integral} that
\begin{equation} \label{E:W-concave}
	w(r)r\le \int_{0}^r w \le G^{-1}(r)
		\quad\text{for $r\in(0,\infty]$}.
\end{equation}
Now, we define $a\colon(0,\infty)\to[0,\infty]$ by
\begin{equation*}
	a(\tau) = w\bigl(h_*(\tau)\bigr)
		\quad\text{for $\tau\in(0,\infty)$}.
\end{equation*}
Then, as $h_*$ is non-increasing and $w$ is nonincreasing, $a$ is non-decreasing and $A(t)=\int_{0}^{t}a$ for $t\in[0,\infty]$ is a Young function.
We claim that $f\in L^A$.
Indeed,
\begin{equation*}
	\int_{0}^\infty A\left( \frac{\abs{f}}{\lambda} \right)
		= \int_{0}^\infty A(\abs{h})
		= \int_{\{h_*\ne0\}} w(h_*)h_*
		\le \int_{0}^\infty G^{-1}(h_*)
		\le 1,
\end{equation*}
where we used Fubini's theorem and inequalities~\eqref{E:W-concave} and \eqref{E:Wh-one}.
In conclusion, $\varrho(f/\lambda)\le 1$ and hence, by the definition of the Luxemburg norm, $\nrm{f}_{L^A}\le \lambda=2\nrm{f}_{\Lambda^E}$ proving that $f\in L^A$.

In the rest of the proof, we show that $a$ obeys condition~\eqref{E:union-function-inequality} with $N_1\le 1$.
First, observe that
\begin{equation} \label{E:g-invers-ineq}
	g^{-1}\left( \frac{1}{w(r)} \right)
		= g^{-1}\Bigl(g\bigl(G^{-1}(r)\bigr)\Bigr)
		\le G^{-1}(r)
	\quad\text{for $r\in[0,\infty]$},
\end{equation}
where $g^{-1}$ is the left-continuous inverse of $g(t)=E_\#(t)/t$.
Indeed, for $r\in(0,\infty)$, the equality in \eqref{E:g-invers-ineq} follows by definition~\eqref{E:w-def} of $w$, and the inequality is due to~\eqref{E:comp-lc}.
Next, $w(0)=\infty$ by definition, and $g^{-1}(0)=0$ and $G^{-1}(0)=0$ by the left-continuity.
Thus, \eqref{E:g-invers-ineq} holds also for $r=0$.
Finally, by definition, $w$ is left-continuous at infinity and hence~\eqref{E:g-invers-ineq} is satisfied also for $r=\infty$ by taking the limit as $r\to\infty$ due to left-continuity of $g^{-1}$ and $G^{-1}$.
Consequently,
\begin{equation*}
	N_1
		= \int_{0}^{\infty} g^{-1}\left( \frac{1}{a(t)} \right)\d t
		= \int_{0}^{\infty} g^{-1}\left( \frac{1}{w\bigl(h_*(t)\bigr)} \right)\d t
		\le \int_{0}^{\infty} G^{-1}\bigl(h_*(t)\bigr)\d t
		\le 1,
\end{equation*}
where the last inequality is due to~\eqref{E:Wh-one} and condition~\eqref{E:union-function-inequality} is satisfied.
\end{proof}

We can finally formulate the principal results of this section, namely that every strong/weak Orlicz space can be expressed as a union/intersection of smaller/larger Orlicz spaces.
The first theorem for the strong Orlicz space follows immediately from Theorems~\ref{T:embedding-orlicz-into-endpoint-Lorentz} and~\ref{T:union-absolutely-continuous}.
The result for the weak Orlicz space then follows by duality arguments.

\begin{theorem} \label{T:lambda-union}
Let $E$ be a quasi-convex function. Then
\begin{equation*}
	\Lambda^E
		= \bigcup \set*{L^A: \text{$A$ is a Young function such that $L^A \hra \Lambda^E$}}.
\end{equation*}
\end{theorem}

\begin{theorem}\label{T:marcinkiewicz-intersection}
Let $F$ be a quasi-convex function. Then
\begin{equation*}
    M^F = \bigcap\set*{L^B: \text{$B$ is a Young function such that $M^F\hra L^B$}}.
\end{equation*}
\end{theorem}

\begin{proof}
The inclusion ``$\subset$'' is clear.
Conversely, let $g\in\bigcap\{L^B: M^F\hra L^B\}$. We need that
$g\in M^F$. To this end, it suffices that $\int_0^\infty fg<\infty$
for every $f\in \Lambda^{\widetilde F}=(M^F)'$
\citep[see][Chapter~1, Lemma~2.6]{BS}. Since, by Theorem~\ref{T:lambda-union},
$\Lambda^{\widetilde F}=\bigcap\{L^A:L^A\hra\Lambda^{\widetilde F}\}$,
there is $L^A$ such that $f\in L^A$ and $L^A\hra\Lambda^{\widetilde F}$.
Setting $B=\widetilde A$, we get by duality that $M^F\hra L^B$
and hence, by the assumption that $g\in L^B$. Therefore, by the definition
of the associate space \eqref{E:associate}, we conclude that
\begin{equation*}
	\int_0^\infty fg
		\le \nrm{f}_{L^A}\nrm{g}_{(L^A)'}
		\le C\nrm{f}_{L^A}\nrm{g}_{L^B} <\infty
\end{equation*}
for a certain constant $C$.
\end{proof}

\section{The Principal Alternative}\label{S:principal-alternative}

In this chapter, we will prove Theorem~\ref{T:intro-principal-alternative-for-spaces}.
We shall first define rearrangement-invariant spaces and list some of their basic properties.
The standard reference is \citep[Chapter 2]{BS}.
We will restrict our attention to spaces of measurable functions over $\sigma$-finite, non-atomic measurable spaces.
A fair portion of the theory may be considered in more general so-called resonant measurable spaces, or even arbitrary $\sigma$-finite spaces.
However, as no applications of such a general approach are known to us, we stick to the simplified setting.

We fix, for the entire chapter, two $\sigma$-finite non-atomic measurable spaces $\RM$ and $\SN$.
When only one of the spaces is needed, we stick to $\RM$.
When dealing with operators acting on and taking values in spaces of measurable functions, $\RM$ will play the role of the ``domain'' and $\SN$ of the~``target''.

\paragraph{Rearrangement-invariant Banach function spaces}

Denote by $\MRM$ the set of all measurable functions $f\colon\RR\to[-\infty,\infty]$, by $\measurablep\RM$ the set of all $f\in\MRM$ for which $f\ge 0$ \ae and by $\measurable_0\RM$ the set of all $f\in\MRM$ for which $\abs{f}<\infty$ \ae.
Let $\rho\colon\measurablep\RM\to[0,\infty]$ be a map.
Let $f,g,f_n\in\measurablep\RM$, $n\in\N$, $\lambda\ge 0$, and consider the following properties
\begin{enumerate}[label={\rm (P\arabic*)}]
	\item\label{en:norm}
		$\rho(f+g)\le \rho(f) + \rho(g) $; \quad
		$\rho(\lambda f)=\lambda \rho(f)$; \quad
		if $\rho(f)=0$, then $f=0$;
	\item if $f\le g$ \ae, then $\rho(f)\le \rho(g)$;
	\item if $f_n\nearrow f$ \ae, then $\rho(f_n)\nearrow f$;
	\item if $E\subset \RR$ is of finite measure, then $\rho(\chi_E)<\infty$;
	\item if $E\subset \RR$ is of finite measure, then there is some $C_E>0$ such that $\nrm{f}_{L^1(E)}\le C_E\rho(f)$;
	\item\label{en:ri} if $f^*=g^*$, then $\rho(f)=\rho(g)$.
\end{enumerate}

If $\rho$ satisfies the properties \ref{en:norm}-\ref{en:ri}, we shall call it a \emph{rearrangement-invariant Banach function norm} (\emph{\ri norm} for short).
The collection $X(\rho)$ of all $f\in\MRM$ such that $\rho(\abs{f})<\infty$ will be called a \emph{rearrangement-invariant Banach function space} (\emph{\ri space} for short).
The space $X(\rho)$ is linear and endowed with the complete norm
\begin{equation}\label{E:r-i-norm-extension}
	\nrm{f}_{X\RM} = \rho(\abs{f})
		\quad\text{for $f\in\MRM$.}
\end{equation}
When we say that $X\RM$ is an \ri space, we mean that there is some \ri norm $\rho$ such that $X\RM=X(\rho)$, in which case we consider $\nrm{\cdot}_{X\RM}$ defined as in \eqref{E:r-i-norm-extension}.
Often, when the underlying measure space is obvious from the context, we will write $X$ instead of $X\RM$ or we simply say ``$X$ is an \ri space over $\RM$''.
When $\RR$ is an interval $(a,b)$ and $\mu$ is the one-dimensional Lebesgue measure, we will write just~$X(a,b)$.

If $A$ is a quasi-convex function, then one may define the Lorentz, Orlicz and Marcinkiewicz spaces $\Lambda^A\RM$, $L^A\RM$, $M^A\RM$ as the obvious modification of the case $\RR=(0,\infty)$.
The Marcinkiewicz space $M^A\RM$ is an \ri space.
Moreover, the Lorentz space $\Lambda^A\RM$ and the Orlicz space $L^A\RM$ may be equivalently renormed so as to become \ri spaces.
If $A$ is a Young function, then the space $L^A\RM$ is an \ri space.

\paragraph{Embeddings and order}

Recall that if $X$, $Y$ are \ri spaces over $\RM$, then one has $X\subset Y$ if and only if $X\hra Y$.
In particular, the equality of sets $X=Y$ implies that the norms of $X$ and $Y$ are equivalent.
Whenever we refer to an ordering on some family of \ri spaces, we refer to ``$\subset$'', which is the same as ``$\hra$''.
In particular, we consider \ri spaces (over the same measure space) equal, if they are equal as sets, which means they also have equivalent norms.
We use the terms \emph{largest} and \emph{smallest} space when referring to the greatest and the least element of some family of \ri spaces.
When we want to somehow refer to both the properties of being smallest and largest simultaneously, we use the word \emph{optimal}.

\paragraph{Associate spaces}

For an \ri space $X\RM$, the functional $\nrm{\cdot}_{X'\RM}\colon \measurablep\RM\to[0,\infty]$ given by
\begin{equation}
	\nrm{f}_{X'\RM}
		= \sup\set[\big ]{\nrm{fg}_{L^1}: \nrm{f}_{X\RM}\le 1}
\end{equation}
is an \ri norm called \emph{associate norm} and the corresponding space $X'\RM$ is an \ri space called the \emph{associate space}.
The associate space of $X'\RM$ coincides with $X$, \ie $(X')'\RM=X\RM$ with equal norms.
If $X$, $Y$ are \ri spaces, then
\begin{equation}
	X \hra Y
		\quad\text{if and only if}\quad
	Y' \hra X'.
\end{equation}

\paragraph{The Luxemburg representation theorem}

The Luxemburg representation theorem, \cf \citep[Chapter 2, Theorem 4.10]{BS}, asserts that for any \ri space $X$ over $\RM$, there is a unique \ri space over $(0,\MR)$, denoted by $\widebar{X}$ and called the \emph{Luxemburg representation space} of $X$, such that
\begin{equation} \label{E:luxemburg-repr}
	\nrm{f}_{X\RM}
		= \nrm{f^*}_{\widebar{X}(0,\MR)}
	\quad\text{for all $f\in\MRM$}.
\end{equation}

It can be easily seen that $L^A(0,\MR)$, $\Lambda^A(0,\MR)$ and $M^A(0,\MR)$ are the Luxemburg representation spaces of $L^A\RM$, $\Lambda^A\RM$ and $M^A\RM$ respectively.

The strength of the Luxemburg representation theorem is that it enables to extend many results that hold for measurable functions on $(0,\infty)$ to general non-atomic $\sigma$-finite measure spaces.
Due to this fact, some results from Section~\ref{S:unions} extend in an obvious way to spaces over non-atomic $\sigma$-finite measure spaces.
For instance, embeddings between standard, weak and strong Orlicz spaces~\eqref{E:sandwich} hold on $\RM$ without any modification.
Embeddings between two spaces of the same type~\eqref{E:embeddings} hold as well when $\MR=\infty$, whereas the condition $B\prec A$ needs to be replaced by the same inequality near infinity when $\MR<\infty$.

Moreover, the final results of Section \ref{S:unions}, namely Theorems~\ref{T:lambda-union} and~\ref{T:marcinkiewicz-intersection} hold in the setting of a general $\sigma$-finite non-atomic space as well.
We will also use the Luxemburg representation to obtain from the aforementioned results proofs of the two subsequent Theorems~\ref{T:mordor-lorentz} and~\ref{T:mordor-marcinkiewicz}.

First, we need to observe an equivalence of three embeddings, namely
\begin{equation} \label{E:zesilovace}
	\Lambda^A\RM\hra L^B\RM
	\quad\text{iff}\quad
	L^A\RM\hra L^B\RM
	\quad\text{iff}\quad
	L^A\RM\hra M^B\RM.
\end{equation}
Indeed, the first and the last embeddings in \eqref{E:zesilovace} follow from $L^A\hra L^B$ via trivial embeddings~\eqref{E:sandwich}.
Conversely, using fundamental relation~\eqref{E:fundamental}, testing the embeddings on characteristics functions yields $A_\#\prec B_\#$ (globally if $\MR=\infty$ or near zero if $\MR<\infty$) that is equivalent to $B\prec A$ (globally, resp.\ near infinity) which therefore yields $L^A\hra L^B$.

\begin{theorem}[non-existence of a largest Orlicz space]\label{T:mordor-lorentz}
Let $E$ be a quasi-convex function and let $\FF$ be a set of measurable functions over $\RM$ satisfying $\Lambda^E\RM\subset\FF$.
If $L^E\RM\not\subset\FF$, then no largest Orlicz space contained in $\FF$ exists.
\end{theorem}

\begin{proof}
All the spaces are considered over $\RM$ unless otherwise specified.
We proceed by contradiction.
Suppose that there is the largest Orlicz space, say $L^F$, contained in $\FF$.
We show that $\Lambda^E\hra L^F$.
To this end, let $f\in\Lambda^E$.
As $\Lambda^E(0,\MR)$ represents $\Lambda^E$, we have $f^*\in \Lambda^E(0,\infty)$.
Since, by Theorem~\ref{T:lambda-union}, $\Lambda^E(0,\infty)$ is a union of all smaller Orlicz spaces, there is a quasi-convex function $A$ such that $L^A(0,\infty)\hra\Lambda^E(0,\infty)$ and $f^*\in L^A(0,\infty)$.
Consequently, $f\in L^A$ and $L^A\hra\Lambda^E$.
Therefore, as $L^A\subset\FF$ and $L^F$ is the largest among Orlicz spaces, it follows that $L^A\subset L^F$ and hence $f\in L^F$.
Finally, $\Lambda^E\hra L^F$ is equivalent to $L^E\hra L^F$ as shown in~\eqref{E:zesilovace} and, consequently, $L^E\subset\FF$ which is a contradiction.\looseness=-1
\end{proof}

The following result is a dual version employing Theorem~\ref{T:marcinkiewicz-intersection} and its proof is similar.

\begin{theorem}[non-existence of a smallest Orlicz space]\label{T:mordor-marcinkiewicz}
Let $F$ be a quasi-convex function and let $\FF\subset M^F\RM$.
If $\FF\not\subset L^F\RM$, then there is no smallest Orlicz space containing $\FF$.
\end{theorem}

\begin{proof}
Assume, by contradiction, that $L^E$ is the smallest Orlicz space containing~$\FF$.
We claim that $L^E\hra M^F$.
Indeed, if $B$ is a quasi-convex function such that $M^F\hra L^B$, then $\FF\subset L^B$ and, as $L^E$ is the smallest Orlicz space containing $\FF$, we have that $L^E\subset L^B$.
Taking the intersection over all such $B$, Theorem~\ref{T:marcinkiewicz-intersection} with Luxemburg representation theorem yield that $L^E\subset M^F$ and hence $L^E\hra M^F$.
Using \eqref{E:zesilovace}, $L^E\hra L^F$ and therefore $\FF\subset L^F$ which is impossible.
\end{proof}

\paragraph{The fundamental function, fundamental Orlicz space and fundamental level}

As $\RM$ is non-atomic, for any $t\in [0,\MR]$ there exists a measurable $E\subset\RR$ with $\mu(E)=t$, see~\eg~\citep[Lemma~2]{Hal:47}.
Hence, it is possible to define
\begin{equation}
	\varphi_X(t)=\nrm{\chi_E}_X
		\quad\text{for $t\in[0,\MR]$},
\end{equation}
where $E\subset\RR$ is an arbitrary measurable subset with $\mu(E)=t$. If $\MR<\infty$, we set
\begin{equation}
	\varphi_X(t) = \varphi_X\bigl(\MR\bigr)
		\quad\text{for $t\in[\MR,\infty]$}.
\end{equation}
The function $\varphi_X$ is called the \emph{fundamental function of the space $X$}.
We shall say that two \ri spaces over $\RM$ are on the \emph{same fundamental level} if their fundamental functions are equivalent, \ie if they mutually dominate each other up to a multiplicative constant.
In other words, if $X$ and $Y$ are on the same fundamental level, their norms are indistinguishable on characteristic functions and in this sense $X$ and $Y$ are ``close'' to each other.
Fundamental functions of an \ri space $X$ and its associate space $X'$ are related by the formula
\begin{equation}\label{E:fundamental-associate}
	\varphi_X(t)\varphi_{X'}(t) = t
		\quad\text{for $t\in[0,\mu(R)]$.}
\end{equation}

For any \ri space $X$, the fundamental function $\varphi_X$ is always non-decreasing, continuous on $(0,\infty]$ with $\varphi_X(0)=0$.
Moreover, the function $t\mapsto {t}/{\varphi_X(t)}$ is non-decreasing and therefore, there exists a unique quasi-convex function $A$ such that $\smash{A^{-1}_\#=\varphi_X}$.
Thus, we define the \emph{fundamental Orlicz space} of $X$ by $L(X)=L^A\RM$ for this $A$ and analogously $\Lambda(X)=\Lambda^A\RM$ and $M(X)=M^A\RM$, the \emph{fundamental Lorentz} and \emph{fundamental Marcinkiewicz} spaces, respectively.
Recalling equalities~\eqref{E:fundamental}, we infer that
\begin{equation} \label{E:fundamental-ri}
	\varphi_{\Lambda^A}
		= \varphi_{L^A}
		= \varphi_{M^A}
		= A^{-1}_\#
\end{equation}
and the spaces $X$, $L(X)$, $\Lambda(X)$ and $M(X)$ are on the same fundamental level.
Note that embeddings \eqref{E:embeddings} and \eqref{E:zesilovace} are fully governed by the inequalities between the involved quasi-convex functions which equivalently translates via \eqref{E:fundamental-ri} and \eqref{E:domination} to the inequalities between fundamental functions, namely
\begin{equation} \label{E:fundamental-inequalities}
	L^A\hra L^B
	\quad\text{iff}\quad
	\varphi_{L^B} \lesssim \varphi_{L^A}
\end{equation}
\etc.

For a Luxemburg representation space $\widebar{X}$ of $X\RM$, one has $\varphi_X=\varphi_{\widebar{X}}$.
This fact, together with \cite[Chapter 2, Theorem 5.13]{BS} yields the following sandwich
(\cf embeddings~\eqref{E:sandwich})
\begin{equation} \label{E:Lorentz-Marcinkiewicz-Sandwich}
	\Lambda(X)\hra X \hra M(X)
\end{equation}
with embedding constants independent of $X$.

While the set-wise relation of $X$ with $\Lambda(X)$ and $M(X)$ is immutable, the relation of $X$ and $L(X)$ could be, in principle, arbitrary.
It may happen that $X\hra L(X)$, $L(X)\hra X$ or even that $X$ and $L(X)$ are incomparable.

Now, after all the necessary preliminaries, we are ready to prove Theorem~\ref{T:intro-principal-alternative-for-spaces}.

\begin{proof}[Proof of Theorem~\ref{T:intro-principal-alternative-for-spaces}]
Let $L(X)\subset X$ and assume that $A$ is a quasi-convex function obeying $L^A\subset X$.
Then $L^A\hra X$ and $\varphi_{L(X)}=\varphi_X\lesssim\varphi_{L^A}$.
Hence, by~\eqref{E:fundamental-inequalities}, we also have $L^A\hra L(X)$ and $L(X)$ is the largest among all Orlicz spaces contained in $X$.
If $L(X)\not\subset X$, then we recall that $\Lambda(X)\subset X$ always holds and the assertion follows by Theorem~\ref{T:mordor-lorentz} applied to $\FF=X$.
The proof of \ref{en:PA-spaces-target} follows analogously, using Theorem~\ref{T:mordor-marcinkiewicz}.
\end{proof}

\section{Principal Alternative for Sobolev Embeddings}\label{S:Principal-alternative-for-sobolev-embeddings}

In this section, we shall focus on the optimality of Orlicz spaces in Sobolev embeddings.
We shall consider here the full-gradient Sobolev spaces $W^mX$, as opposed to those mentioned in the introduction, denoted $V^m_0X$ and involving only the highest-order derivatives and functions vanishing on the boundary of the underlying domain.
We begin with formal definitions of the spaces in question.

\paragraph{Sobolev spaces}

Let $\Omega\subset\R^n$, $n\in\N$, be an open set and let $X(\Omega)$ be an \ri space.
The \emph{Sobolev space} $W^mX(\Omega)$ of order $m\in\N$ is defined as a collection of all weakly differentiable functions $u\colon\Omega\to\R$ such that $\abs{\nabla^k u}\in X(\Omega)$ for every $k=0,1,\dots,m$, where $\nabla^ku=({\partial^k u}/{\partial x_1^{\alpha_1}\cdots\partial x_k^{\alpha_k}})_{\abs{\alpha}=k}$, $\nabla^0u=u$, and $\abs{\,\cdot\,}$ stands for the $n$-dimensional Euclidean norm.
We furnish the space $W^mX(\Omega)$ with the norm $\nrm{u}_{W^mX(\Omega)}=\sum_{k=0}^m\nrm[\big ]{\abs{\nabla^k u}}_{X(\Omega)}$.

In the particular case when $X(\Omega)=L^A(\Omega)$ is an Orlicz space, we call $W^mX(\Omega)=W^mL^A(\Omega)$ an \emph{Orlicz--Sobolev} space.
We also write just $W^mX$ when the set $\Omega$ is clear from the context.

\paragraph{Sobolev embeddings}

We will consider Sobolev embeddings of the form
\begin{equation} \label{E:sobolev}
	W^m X(\Omega) \hra Y(\Omega),
\end{equation}
where $m,n\in\N$, $m\le n-1$, $\Omega$ is a bounded open set in $\R^n$ equipped with the Lebesgue measure, and $X(\Omega),Y(\Omega)$ are rearrangement-invariant spaces.
For simplicity, we will restrict ourselves here only to those sets $\Omega$ which have Lipschitz boundary, as our aim is to show a prototypical use of our ``principal alternatives''.
Variants for other types of embeddings are available, but some of them are rather technical.
We will briefly discuss this matter in Remark~\ref{R:other-sobolev-embeddings}.
Moreover, we shall assume, without loss of generality, that $\abs{\Omega}=1$.
This restriction is adopted for technical convenience only and has no impact on the spaces appearing in the embeddings.

\paragraph{Optimal \ri spaces}

Our major concern is the optimal form of Sobolev embeddings within \ri spaces, and within Orlicz spaces (both on the domain and on the target sides).
For \ri spaces, the available theory is reasonably complete, see \eg~\citep{Edm:00,Ker:06}, or \citep{Cia:15}.
In particular, it is known that the optimal \ri target space and the optimal domain \ri domain space always exist and can be (almost) explicitly described.
Namely, given an \ri~space $Y$, the optimal \ri~domain space $X$ in embedding~\eqref{E:sobolev} exists and obeys
\begin{equation} \label{E:optimal-ri-domain}
	\nrm{u}_{X(\Omega)}
		= \sup_{h} \nrm*{\int_{t}^{1} h(s)s^{\mn-1}\dd s}_{\widebar Y(0,1)}
	\quad\text{for $u\in\MM(\Omega)$},
\end{equation}
in which the supremum is taken over all $h\in\MM_+(0,1)$ such that $h^*=u^*$.
Conversely, given an \ri space $X$, the optimal \ri target space $Y$ in \eqref{E:sobolev} exists and satisfies
\begin{equation} \label{E:optimal-ri-range}
	\nrm{u}_{Y'(\Omega)}
		= \nrm[\big ]{t^\mn u^{**}(t)}_{\widebar X'(0,1)}
	\quad\text{for $u\in\MM(\Omega)$}.
\end{equation}
For more details, see \eg \cite[Theorem~A and Theorem~3.3]{Ker:06}.
\medskip

When restricted to Orlicz spaces, \citet{Cia:96} proved that, to a given Orlicz space $L^A$, the optimal Orlicz target space $L^B$ in the embedding
\begin{equation}\label{E:sobolev-embedding-vanishing-orlicz}
	W^{m}L^A(\Omega)\hra L^B(\Omega)
\end{equation}
for the first order $m=1$ always exists and possesses an explicit form.
This result was later extended to embeddings of any order \citep{Cia:06}, to boundary trace embeddings \citep{Cia:10}, to first-order Gaussian--Sobolev embeddings \citep{Cia:09}, to higher-order Gaussian--Sobolev embeddings and more general embeddings on probability spaces \citep{Cia:15}, to embeddings of Sobolev spaces in spaces endowed with Frostman measures \citep{Cia:20-b}, and to embeddings of fractional Orlicz--Sobolev spaces \citep{Alb:21}.

On the domain side, the situation is more subtle.
\citet[Theorem~4.3]{Pic:98} showed (see also~\cite[Theorem~4.5]{Pic:02}) that to the given target $L^B=\exp L^{n'}$, no optimal Orlicz domain space $L^A$ in embedding~\eqref{E:sobolev-embedding-vanishing-orlicz} with $m=1$ exists.
The proof departs from a given Young function $A$ such that embedding~\eqref{E:sobolev-embedding-vanishing-orlicz} holds, and a rather technical construction is presented that produces another Young function, say $A_1$, such that $L^{A_1}$ is strictly larger than $L^A$, and yet the embedding $W^{m}L^{A_1}(\Omega)\hra Y(\Omega)$ still holds.
\citet{CiP:98} used similar construction to obtain an analogous result, this time when the prescribed target space is $L^B=L^\infty$.

The next question, which was successfully studied in this direction, was for which general Marcinkiewicz (or weak Orlicz) spaces $Y$ there exists the optimal Orlicz space $L^A$ in the embedding
\begin{equation}\label{E:sobolev-embedding-marcinkiewicz}
	W^{m}L^A(\Omega)\hra Y(\Omega).
\end{equation}
A complete characterisation of such spaces is available due to~\citet{Mus:16}, and \citet{Cia:19} extended this result to the case when $Y$ is an Orlicz space.
In all the above-mentioned works, the ``negative'' result, \ie the non-existence of optimal domain Orlicz space, was proved in a similar way, namely using a technical ad hoc construction which differed from case to case and which in each case had to be specifically tailored in order to fit the particular situation at hand.
Such a situation of course is not ideal and calls for a simple comprehensive treatment based on some universal method.

We are going now to offer such an approach as one of the possible variants of the principal alternative.
The main point is that no technical construction is needed any longer.
We take a lateral point of view, and instead of looking at a particular case at hand, we consider general relations between groups of function spaces.
An important ingredient is the knowledge of optimal \ri~spaces.
Among applications, we shall present new proofs of known results (Theorem~\ref{T:sobolev-to-orlicz}) as well as solutions to open problems (Theorem~\ref{T:bw-alternative}).
In order to achieve the latter, we shall carry out a careful analysis of the relations between fundamental functions of optimal spaces on the domain and the target side of a Sobolev embedding, culminating in Theorem~\ref{T:nonexistence-optimal-orlicz-on-level}, in which an interesting inheritance of the non-existence of an optimal Orlicz domain space is extrapolated from a Marcinkiewicz space along the corresponding fundamental level.

The key initial step is the following result in the spirit of the principal alternative for \ri~spaces (Theorem~\ref{T:intro-principal-alternative-for-spaces}), specified for Sobolev embeddings.

\begin{theorem}[principal alternative for Sobolev embeddings] \label{T:principal-alternative-sobolev-embeddings}
Let $m,n\in\N$, let $\Omega\subset\R^n$ be a bounded domain with Lipschitz boundary, and let $X(\Omega)$ and $Y(\Omega)$ be \ri~spaces.
\begin{enumerate}
	\item\label{en:PA-sobolev-embeddings-domain} If $X(\Omega)$ is the largest \ri~space rendering embedding~\eqref{E:sobolev} true, then either $L(X)\subset X$ and $L^A=L(X)$ is the largest Orlicz space in
\begin{equation} \label{E:embedding-generic-domain}
    W^{m}L^A(\Omega) \hra Y(\Omega),
\end{equation}
or no largest Orlicz space $L^A$ rendering \eqref{E:embedding-generic-domain} true exists.
	\item\label{en:PA-sobolev-embeddings-target} If $Y$ is the smallest \ri~space rendering embedding~\eqref{E:sobolev} true, then either $Y\subset L(Y)$ and $L^B=L(Y)$ is the smallest Orlicz space in
\begin{equation} \label{E:embedding-generic-target}
    W^{m}X(\Omega) \hra L^B(\Omega),
\end{equation}
or no smallest Orlicz space $L^B$ rendering \eqref{E:embedding-generic-target} true exists.
\end{enumerate}
\end{theorem}

\begin{proof}
Let us show~\ref{en:PA-sobolev-embeddings-domain}.
First observe that, for \ri spaces $X$ and  $Y$, one has $W^mX(\Omega)\hra W^mY(\Omega)$ if and only if $X(\Omega)\hra Y(\Omega)$.
The detailed proof of this is given in~\citep[Proposition~5.5]{Cia:19} for Orlicz spaces, and its extension to all \ri spaces is straightforward.
Therefore, by the optimality of $X(\Omega)$ among \ri~spaces in embedding~\eqref{E:sobolev}, we have $W^{m}L^A(\Omega) \hra Y(\Omega)$ if and only if $L^A(\Omega)\hra X(\Omega)$.
Hence, $L^A(\Omega)$ is the largest Orlicz space in $W^{m}L^A(\Omega) \hra Y(\Omega)$ if and only if $L^A(\Omega)$ is the largest Orlicz space contained in $X(\Omega)$.
Consequently, the assertion follows from part~\ref{en:PA-spaces-domain} of Theorem~\ref{T:intro-principal-alternative-for-spaces}.

The proof of \ref{en:PA-sobolev-embeddings-target} follows from part~\ref{en:PA-spaces-target} of Theorem~\ref{T:intro-principal-alternative-for-spaces} in an analogous way.
\end{proof}

We shall now present a key technical result on the fundamental function of a function space with an operator-induced norm. It does not concern Sobolev embeddings directly, but will be applied to them later.
It is certainly of independent interest and is likely to be applicable to the existence of optimal domain spaces in many other tasks as well.

\begin{theorem}[fundamental function of an operator-induced space]\label{T:fundamental-optimal-domain-hardy}
Let $\alpha\in(0,1)$, $\beta\in(0,\infty)$ and either $L=1$ or $L=\infty$.
For an \ri~space $Y(0,L)$ let $X(0,L)$ be defined by
\begin{equation*}
	\nrm{f}_{X(0,L)}
		= \sup_{h}\nrm*{\int_{\tau^{\beta}}^{L}h(s)s^{\ai}\,\d s}_{Y(0,L)}
	\quad\text{for $f\in\MM(0,L)$},
\end{equation*}
where the supremum is extended over all functions $h\in\MM(0,L)$ such that $h^*=f^*$.
Then $X(0,L)$ is an~\ri~space, and one has
\begin{equation} \label{E:fundamental-optimal-domain-hardy}
	\vpx(t)\approx t\sup_{s\in(t,\infty)}\vp_Y\bigl(s^{\ib[i]}\bigr)s^{\ai}
		\quad\text{for $t\in(0,L)$.}
\end{equation}
\end{theorem}

\begin{proof} We present the proof only for $L=1$. The proof for $L=\infty$ is almost the same, even easier.

The fact that $X$ is an~\ri~space is classical and can be proved in the same way as in~\citep[Theorem~3.3]{Ker:06}.
We will show the inequality ``$\gtrsim$'' in~\eqref{E:fundamental-optimal-domain-hardy}.
Fix $t\in(0,1)$.
By the boundedness of the dilation operator on $Y$ \citep[see \eg][Chapter~3, Proposition~5.11]{Ben:80}, we get
\begin{align*}
	\vpx(t)
		& = \nrm{\chi_{[0,t)}}_{X(0,L)} \ge \nrm*{\int_{\tau^{\beta}}^{1} \chi_{(0,t)}(s)s^{\ai}\dd s}_{Y(0,1)}
			\ge \nrm*{\chi_{(0,\frac{t}{2})}(\tau^\beta) \int_{\tau^\beta}^ts^{\ai}\dd s}_{Y(0,1)}
        \\
    & \ge \int_{{t}/{2}}^ts^{\ai}\dd s\;
						\nrm*{\chi_{(0,\frac t2)}\bigl(\tau^\beta\bigr)}_{Y(0,1)}
    	\approx t^{\alpha}\vpy\bigl(t^{\ib[i]}\bigr),
\end{align*}
showing that
\begin{equation*}
	\frac{\vpx(t)}{t}
		\gtrsim t^{\ai} \vpy\bigl(t^{\ib[i]}\bigr)
	\quad\text{for every $t\in(0,1)$}.
\end{equation*}
Since the function $t\mapsto {\vpx(t)}/{t}$ is nonincreasing, we can enhance the last estimate to
\begin{equation*}
	\frac{\vpx(t)}{t}
		\gtrsim \sup_{s\in(t,\infty)} s^{\alpha-1}\vpy\bigl(s^{\ib[i]}\bigr)
	\quad\text{for every $t\in(0,1)$},
\end{equation*}
which yields the desired inequality.
Note that the supremum is in fact attained on $(t,1)$ as $\vpy$ is, by definition, constant on $[1,\infty)$.

Let us focus on the converse inequality in~\eqref{E:fundamental-optimal-domain-hardy}.
Note that $\vpy$ is differentiable \ae and for any measurable $f$, one has
\begin{equation} \label{E:lambda-norm-eq}
	\nrm{f}_{\Lambda(Y)}
		= \int_{0}^\infty \vpy(f_*)
		= \vpy(0_+)\nrm{f}_\infty + \int_{0}^{1} f^* \vpy'.
\end{equation}
Thus, by fundamental embedding~\eqref{E:Lorentz-Marcinkiewicz-Sandwich} and Fubini's theorem, we have for any $h\ge 0$ that
\begin{align*}
	\nrm*{\int_{\tau^{\beta}}^{1} h(s)s^{\ai}\dd s}_{Y(0,1)}
		& \lesssim \nrm*{\int_{\tau^{\beta}}^{1} h(s)s^{\ai}\dd s}_{\Lambda(Y)(0,1)}
				\\
		& \lesssim \vpy(0_+) \nrm*{\int_{\tau^{\beta}}^{1} h(s)s^{\ai}\dd s}_{L^\infty(0,1)}
				+ \int_{0}^{1} \int_{\tau^{\beta}}^{1} h(s)s^{\alpha-1}\dd s\,\vpy'(\tau)\dd\tau
				\\
		& = \vpy(0_+) \int_{0}^{1} h(s)s^{\ai}\dd s
				+ \int_{0}^{1} \vpy\bigl(s^{\ib[i]}\bigr) h(s)s^{\alpha-1} \dd s
				\\
		& \approx \int_{0}^{1} \vpy\bigl(s^{\ib[i]}\bigr) h(s)s^{\alpha-1} \dd s
			\le \int_{0}^{1} h(s) \sup_{\tau\in(s,\infty)}
						\vpy\bigl(\tau^{\ib[i]}\bigr)\tau^{\alpha-1} \dd s.
\end{align*}
Hence, by the Hardy--Littlewood inequality (see \eg~\citep[Theorem~II.2.2]{BS}), one has, for any $t\in(0,1)$,
\begin{align*}
	\vpx(t)
		& = \sup_{h^*=\chi_{[0,t)}}
					\nrm*{\int_{\tau^{\beta}}^{1} h(s)s^{\ai}\dd s}_{Y(0,1)}
				\\
    & \le \sup_{h^*=\chi_{[0,t)}}
					\int_{0}^{1} h(s) \sup_{\tau\in(s,\infty)}
						\vpy\bigl(\tau^{\ib[i]}\bigr)\tau^{\alpha-1} \dd s
    	\le \int_{0}^{t} \sup_{\tau\in(s,\infty)}
						\vpy\bigl(\tau^{\ib[i]}\bigr)\tau^{\alpha-1} \dd s.
\end{align*}
Therefore, we shall be done once we prove
\begin{equation}\label{E:sandwich-estimate}
	\int_{0}^{t} \frac{\vp(s)}{s}\d s
		\lesssim \vp(t)
	\quad\text{for $t\in(0,1)$},
\end{equation}
in which $\vp$ denotes the function on the right hand side of \eqref{E:fundamental-optimal-domain-hardy}.
Using the equivalent form
\begin{equation} \label{E:vp-alt}
	\vp(t) = \sup_{s\in(0,\infty)} \vpy(s)\min\set[\big ]{t^\alpha, ts^{\beta(\ai)}},
\end{equation}
we infer that $\vp$ is non-decreasing.
It also follows from equality~\eqref{E:vp-alt} that for any $\sigma\in(0,1)$ and any $k\in\N$, one has $\vp(t\sigma^k) \le \sigma^{\alpha k} \vp(t)$ for $t\in(0,1)$ and therefore
\begin{equation*}
	\int_{0}^{t} \frac{\vp(s)}{s}\d s
		= \sum_{k=0}^\infty \int_{t\sigma^{k+1}}^{t\sigma^k} \frac{\vp(s)}{s}\d s
		\le \sum_{k=0}^\infty \vp(t\sigma^k) \int_{t\sigma^{k+1}}^{t\sigma^k} \frac{\d s}{s}
		\le \vp(t) \log\frac1\sigma \sum_{k=0}^\infty \sigma^{\alpha k}
		= \vp(t) \frac{\log\frac1\sigma}{1-\sigma^\alpha}
\end{equation*}
establishing~\eqref{E:sandwich-estimate}, hence the assertion.
\end{proof}

The following consequence of the preceding theorem, tailored for the specific needs of Sobolev embeddings, will later have a decisive role in solving the open problem concerning the existence of optimal Orlicz domains in Sobolev embeddings into non-Marcinkiewicz spaces.

\begin{corollary}\label{C:fundamental-domain-equality}
Let $m,n\in\N$ be such that $n\ge2$ and $m\le n-1$, and let $\Omega\subset \R^n$ be a bounded domain with Lipschitz boundary.
Suppose that $Y_1$, $Y_2$ are \ri~spaces over $\Omega$ on the same fundamental level, \ie $\vp_{Y_1}\approx\vp_{Y_2}$.
Let $X_j$ be the optimal \ri~domain spaces in the embedding
\begin{equation}
	W^m X_j(\Omega) \hra Y_j(\Omega),
\end{equation}
for $j=1,2$.
Then $X_1$ and $X_2$ are also on the same fundamental level, \ie $\vp_{X_1}\approx\vp_{X_2}$.
\end{corollary}

\begin{proof}
The proof follows immediately by the formula for the norm in the optimal \ri~domain space~\eqref{E:optimal-ri-domain} and by Theorem~\ref{T:fundamental-optimal-domain-hardy}.
\end{proof}

As an application of the principal alternative, we shall now present a new direct and non-constructive proof of a known result concerning the existence of an optimal Orlicz domain space $L^A$ in a~Sobolev embedding
\begin{equation} \label{E:sobolev-to-orlicz-embedding}
	W^m L^A(\Omega)\hra L^B(\Omega),
\end{equation}
where $L^B$ is a~prescribed Orlicz space.
For the original proof, see \citep[Theorem~3.4]{Cia:19}.

\begin{theorem}[optimal Orlicz domain in a Sobolev embedding into an Orlicz space] \label{T:sobolev-to-orlicz}
Let $m,n\in\N$ be such that $n\ge2$ and $m\le n-1$, and let $\Omega\subset\R^n$ be a bounded set with Lipschitz boundary.
Suppose that $B$ is a~Young function. Let $B_{n}$ be a Young function satisfying
\begin{equation*}
    B_n^{-1}(t) \approx t\inf_{s\in(1,t)}B^{-1}(s)s^{\mn-1} \quad\text{near infinity.}
\end{equation*}
The following four statements are equivalent.
\begin{enumerate}
	\item\label{en:sobolev-to-orlicz-embedding}
	The embedding
	$W^mL^{B_n}(\Omega)\hra L^B(\Omega)$ holds;
	\item\label{en:sobolev-to-orlicz-optimal-space}
	The space $L^{A}(\Omega)=L^{B_n}(\Omega)$ is the largest Orlicz space in~\eqref{E:sobolev-to-orlicz-embedding};
	\item\label{en:sobolev-to-orlicz-existence-of-optimal-space}
	There exists a largest Orlicz space $L^{A}(\Omega)$ in~\eqref{E:sobolev-to-orlicz-embedding};
	\item\label{en:sobolev-to-orlicz-index}
	The upper Boyd index\footnote{For the definition, see~\eg \citep[Section~2.2]{Cia:19}. For more details on indices of Orlicz spaces, \cf~\citep{Boy:71}.} $I_{B_n}$ of $B_{n}$ satisfies
	\begin{equation} \label{E:sobolev-index}
			I_{B_n}<\nm.
	\end{equation}
\end{enumerate}
\end{theorem}

\begin{proof}
The fundamental function $\vpx$ of the largest \ri~space $X$ in the embedding
\begin{equation}\label{E:embedding-from-X}
    W^mX(\Omega)\hra L^B(\Omega)
\end{equation}
satisfies
\begin{equation*}
	\vpx(t)
		\approx t \sup_{s\in(t,\infty)} \vp_{L^B}(s) s^{\mn-1}
	\quad\text{for $t\in(0,1)$},
\end{equation*}
as follows by formula~\eqref{E:optimal-ri-domain} and Theorem~\ref{T:fundamental-optimal-domain-hardy}.
Using relations between fundamental and Young functions~\eqref{E:fundamental-ri}, we arrive at $L(X)=L^{B_n}$.
Now we can obtain the equivalence of the first three statements easily from the principal alternative for Sobolev embeddings (part~\ref{en:PA-sobolev-embeddings-domain} of Theorem~\ref{T:principal-alternative-sobolev-embeddings}).
Indeed, we either have $L^{B_n}\subset X$, or $L^{B_n}\not\subset X$.
In the former case, statement~\ref{en:sobolev-to-orlicz-embedding} holds owing to this very inclusion and embedding~\eqref{E:embedding-from-X}, statement~\ref{en:sobolev-to-orlicz-optimal-space} follows from Theorem~\ref{T:principal-alternative-sobolev-embeddings}, and statement~\ref{en:sobolev-to-orlicz-existence-of-optimal-space} is an immediate consequence of~\ref{en:sobolev-to-orlicz-optimal-space}.
In the latter case, statement~\ref{en:sobolev-to-orlicz-embedding} fails by the optimality of $X$ in~\eqref{E:embedding-from-X}, statement~\ref{en:sobolev-to-orlicz-existence-of-optimal-space} fails owing to Theorem~\ref{T:principal-alternative-sobolev-embeddings}, and statement~\ref{en:sobolev-to-orlicz-optimal-space} fails since \ref{en:sobolev-to-orlicz-existence-of-optimal-space} does.
So the first three statements either stand or fall together.

In the rest of the proof, we show that statements~\ref{en:sobolev-to-orlicz-embedding} and~\ref{en:sobolev-to-orlicz-index} are equivalent.
Owing to the combination of~\cite[Theorem~3.1]{Cia:06} and~\cite[Lemma~2]{Cia:00}, we know that~\ref{en:sobolev-to-orlicz-embedding} is equivalent to
\begin{equation} \label{E:b-f}
	B\prec F
		\quad\text{near infinity},
\end{equation}
in which the function $F$ is given by
\begin{equation}
	F(t)= \left(tE^{-1}\left(t^{\frac{n}{n-m}}\right)\right)^{\frac{n}{n-m}}
		\quad\text{and}\quad
	E(t) = \int_{0}^{t} \frac{\widetilde{B_n}(s)}{s^{n/(n-m)+1}}\dd s
		\quad\text{near infinity}.
\end{equation}
It can be verified \citep[\cf~\eg][Eq.~2.19]{Cia:00} that
$F^{-1}(t)I^{-1}(t)\approx t^{1-\mn}$ near infinity,
where $I(t)=t^{\frac{n}{n-m}}E(t)$ for large $t$.
Therefore, condition~\eqref{E:b-f} is equivalent to
\begin{equation} \label{E:i-b}
	\frac{1}{I^{-1}(s)}
		\lesssim B^{-1}(s)s^{\mn-1}
	\quad\text{near infinity}.
\end{equation}
Now, since $I$ is increasing and $1/I^{-1}$ is decreasing, inequality~\eqref{E:i-b} is in fact equivalent to
\begin{equation*}
	\frac{1}{I^{-1}(t)}
		\lesssim \inf_{s\in(1,t)} B^{-1}(s)s^{\mn-1}
	\quad\text{near infinity}.
\end{equation*}
Consequently, by the definition of $B_n$ and due to the relation between inverses of a Young function and its conjugate~\eqref{E:upper-bound-for-product}, we have $\widetilde{B_n}^{-1}\lesssim I^{-1}$ near infinity which, passing to inverses, is equivalent to $I\prec\widetilde{B_n}$.
Finally, this inequality is characterized by the Boyd index condition~\eqref{E:sobolev-index}, see~\cite[Proposition~4.1]{Cia:19}, for instance.
\end{proof}

In Corollary \ref{C:fundamental-domain-equality} we have seen that if two \ri~spaces share the same fundamental function, then so do their optimal \ri~domain partners in a Sobolev embedding.
This fact is just a simple consequence of the explicit expression of $\vpx$ in terms of $\vpy$ given by formula~\eqref{E:fundamental-optimal-domain-hardy}.

Our next aim is to show that an analogous result, in which domain and target spaces swap roles, does \emph{not} hold.
This is in some sense natural since no explicit formula yielding $\vpy$ in terms of $\vpx$ holds.
However, a certain analogue of Corollary~\ref{C:fundamental-domain-equality} is still recoverable if a specific restriction on the (joint) fundamental function of the given (domain) spaces holds, as the next result (or more precisely its corollary) shows.
It may fail in general, though, as will be illustrated in Remark~\ref{R:fundamental}.

\begin{theorem} \label{T:fundamental-target-inequality}
Let $\alpha\in(0,1)$ and $\beta\in(0,\infty)$.
Suppose that $X(0,1)$ is an \ri space such that
\begin{equation} \label{E:fund-cond-target}
	\int_{t}^1 \frac{s^{\ai}}{\vpx(s)} \dd s
		\lesssim  \frac{t^{\alpha}}{\vpx(t)}
	\quad\text{for $t\in(0,1)$.}
\end{equation}
If $Y(0,1)$ is the \ri~space whose associate space $Y'(0,1)$, satisfies
\begin{equation} \label{E:optimal-range-general}
  \nrm{g}_{Y'(0,1)}
		= \nrm*{\tau^{\ai+\ib[i]} g^{**}\bigl(\tau^{\ib[i]}\bigr)}_{X'(0,1)}
	\quad\text{for $g\in\MM(0,1)$},
\end{equation}
then its fundamental function $\vpy$, obeys
\begin{equation} \label{E:fund-target}
  \vpy(t)\approx t^{-\alpha\beta}\vpx\bigl(\tb\bigr)
	\quad\text{for $t\in(0,1)$.}
\end{equation}
\end{theorem}

\begin{proof}
Setting $g=\chi_{[0,t)}$ in~\eqref{E:optimal-range-general} and by a~standard calculation (see~\eg~\cite[Eq.~5.11]{Cia:19} or \cite[Proposition~9.1]{Pic:98}), we get
\begin{equation} \label{E:fund-estimate-target}
  \varphi_{Y'}(t)
		\approx t \nrm[\big ]{\tau^{\ai}\chi_{(\tb,1)}(\tau)}_{X'(0,1)}
	\quad\text{for $t$ near zero.}
\end{equation}
Let us denote by $h_t(\tau) = \tau^{\ai}\chi_{(\tb,1)}(\tau)$ for $\tau,t\in(0,1)$.
Then we have
\begin{equation}\label{E:def-gt}
	h_t^*(s) = \brk[\big ]{s+\tb}^{\ai}\chi_{[0,1-\tb)}(s)
		\quad\text{and}\quad
	h_t^{**}(s) \approx s^{-1}\brk[\big ]{(s+\tb)^{\alpha}-t^{\alpha\beta}}
		\quad\text{for $s\in(0,1)$}.
\end{equation}
Using estimate~\eqref{E:fund-estimate-target} and fundamental embedding~\eqref{E:Lorentz-Marcinkiewicz-Sandwich}, we get
\begin{equation*}
	\varphi_{Y'}(t)
		\gtrsim t \nrm[\big ]{s^{\ai}\chi_{(\tb,1)}(s)}_{M(X')}
    = t \sup_{s\in (0,1)} h_t^{**}(s) \varphi_{X'}(s)
    = t \sup_{s\in (0,1)} s^{-1}\left((s+\tb)^{\alpha}-t^{\alpha\beta}\right) \varphi_{X'}(s).
\end{equation*}
Evaluating the supremum at $s=\tb$, we get $\varphi_{Y'}(t)\gtrsim t^{(\ai)\beta+1}\vpxp\brk{\tb}$ which, employing formula for fundamental function of an associate space~\eqref{E:fundamental-associate}, yields the inequality ``$\lesssim$'' in~\eqref{E:fund-target}.

Conversely, condition~\eqref{E:fund-cond-target} implies that
\begin{equation} \label{E:fund-cond-target-2}
	\int_{\tau}^{1} s^{\ai} \varphi'_{X'}(s)\dd s
		\lesssim \int_{\tau}^{1} \frac{s^{\ai}}{\vpx(s)}\d s
		\lesssim \frac{\tau^{\alpha}}{\vpx(\tau)}
	\quad\text{for $\tau\in(0,1)$}.
\end{equation}
Next, relation~\eqref{E:fund-estimate-target} with embedding to the fundamental Lorentz space~\eqref{E:Lorentz-Marcinkiewicz-Sandwich} and its norm expressed as in~\eqref{E:lambda-norm-eq} give
\begin{align}
	\begin{split}
	\frac{\varphi_{Y'}(t)}{t}
		\lesssim \nrm{h_t}_{\Lambda(X')}
		&	= \varphi_{X'}(0_+)\nrm{h_t}_\infty + \int_0^1 h_t^* \varphi'_{X'}
			\\
    & \le \varphi_{X'}(0_+)t^{\beta(\ai)} + \int_0^{1} \brk[\big ]{s+\tb}^{\ai} \varphi'_{X'}(s)\dd s
	\quad\text{for $t\in(0,1)$}.
	\end{split}
\end{align}
Next, we split the interval of integration at $\tb$ to get
\begin{equation}
	\int_0^{1} \brk[\big ]{s+\tb}^{\ai} \varphi'_{X'}(s)\dd s
		\approx t^{\beta(\ai)} \int_{0}^{\tb} \varphi'_{X'}(s)\dd s
			+ \int_{\tb}^{1} s^{\ai} \varphi'_{X'}(s)\dd s,
\end{equation}
which by the use of~\eqref{E:fundamental-associate} and estimate~\eqref{E:fund-cond-target-2} to $\tau=t^{\beta}$ gives
\begin{align*}
	\frac{1}{\vpy(t)}
		= \frac{\varphi_{Y'}(t)}{t}
    \lesssim t^{\beta(\ai)} \varphi_{X'}(t^\beta) + \int_{\tb}^{1} s^{\ai} \varphi'_{X'}(s)\dd s
		\lesssim \frac{t^{\alpha\beta}}{\vp_{X}(\tb)}
		\quad\text{for $t\in(0,1)$.}
\end{align*}
This finally yields the inequality ``$\gtrsim$'' in~\eqref{E:fund-target}.
\end{proof}

\begin{corollary}\label{C:fundamental-target-equality}
Let $m,n\in\N$ be such that $n\ge2$ and $m\le n-1$, and let $\Omega\subset \R^n$ be a bounded domain with Lipschitz boundary.
Suppose that $X_1$, $X_2$ are \ri~spaces over $\Omega$ on the same fundamental level, \ie $\vp_{X_1}\approx\vp_{X_2}$. Assume that
\begin{equation} \label{E:fund-cond-target-sobolev}
	\int_{t}^1 \frac{s^{\mn-1}}{\vpx(s)} \dd s
		\lesssim  \frac{t^{\mn}}{\vpx(t)}
	\quad\text{for $t\in(0,1)$.}
\end{equation}
Let $Y_j$ be the optimal \ri~domain spaces in the embedding
\begin{equation}
	W^m X_j(\Omega) \hra Y_j(\Omega),
\end{equation}
for $j=1,2$. Then $Y_1$ and $Y_2$ are also on the same fundamental level, \ie $\vp_{Y_1}\approx\vp_{Y_2}$.
\end{corollary}

\begin{proof}
    The proof follows immediately by the formula for the norm in the optimal \ri~target space in~\eqref{E:optimal-ri-range} and by Theorem~\ref{T:fundamental-target-inequality}.
\end{proof}

\begin{remark} \label{R:index-condition}
Note that condition~\eqref{E:fund-cond-target} is in fact equivalent to an explicit growth condition.
Namely, an increasing function $\vp$ obeys condition~\eqref{E:fund-cond-target} if and only if there exist constants $\sigma,c\in(0,1)$ such that
\begin{equation*}
	\varphi(\sigma t) \le c\sigma^{\alpha}\varphi(t)
	\quad\text{for $t\in(0,1)$}.
\end{equation*}
This condition is further closely related to the so-called fundamental indices of \ri~spaces, see~\eg~\cite[Chapter~3, Exercise~14]{BS}.
\end{remark}

\begin{remark}\label{R:fundamental}
The assertion of Corollary~\ref{C:fundamental-target-equality}, as we alluded to before, is not necessarily true if condition~\eqref{E:fund-cond-target-sobolev} is violated.
Indeed, consider the family of Lorentz spaces $L^{\nm,\,q}(\Omega)$ for $q\in[1,\infty]$, determined by the functionals
\begin{equation*}
	\nrm{f}_{\nm,\,q}
		= \nrm*{t^{\mn-\iq} f^{*}(t)}_{L^q(0,1)}.
\end{equation*}
Then, for every such $q$, all these spaces share the same fundamental function, more precisely,
\begin{equation*}
	\vp_{L^{\nm,q}}(t) \approx t^{\frac{m}{n}}
		\quad\text{for $t\in(0,1)$.}
\end{equation*}
Observe that this function violates condition~\eqref{E:fund-cond-target-sobolev}. Now, let $Y_q$ denote the smallest \ri~space in the Sobolev embedding
\begin{equation}\label{E:yq-embedding}
    W^{m}L^{\nm,q}(\Omega) \hra Y_q(\Omega).
\end{equation}
Then, by formula~\eqref{E:optimal-ri-range} for the optimal space, relation~\eqref{E:fund-estimate-target}, the known explicit form of the associate space of $L^{\nm,q}$, and standard calculations, we get for all $q\in[1,\infty]$ that
\begin{equation*}
	\vp_{Y'_{q}}(t)
		\approx t\nrm[\big ]{s^{\mn-1}\chi_{(t,1)}(s)}_{\brk[\big ]{L^{\nm,\,q}}'}
		\approx t \left(1-\log t\right)^{1-\frac{1}{q}}
		\quad\text{near zero},
\end{equation*}
whence $\vp_{Y_{q}}(t)\approx (1-\log t)^{1/q-1}$ near zero.
Therefore, although all the spaces $L^{\nm,q}(\Omega)$, $q\in[1,\infty]$, share the same fundamental function, the fundamental functions of $Y_{q_1}$ and $Y_{q_2}$ are essentially different whenever $q_1,q_2\in[1,\infty]$ are not equal.
\end{remark}

\begin{theorem}\label{T:nonexistence-optimal-orlicz-on-level}
Let $m,n\in\N$ be such that $n\ge2$ and $m\le n-1$, and let $\Omega\subset \R^n$ be a bounded domain with Lipschitz boundary.
Let $Y(\Omega)$ be an \ri~space.
Assume that there is no largest Orlicz space $L^A(\Omega)$ in the Sobolev embedding
\begin{equation} \label{E:emb-MY}
	W^{m}L^A(\Omega) \hra M(Y)(\Omega).
\end{equation}
Then there is no largest Orlicz space $L^A(\Omega)$ in the Sobolev embedding
\begin{equation} \label{E:emb-Y}
	W^{m}L^A(\Omega) \hra Y(\Omega).
\end{equation}
\end{theorem}

\begin{proof}
Let $X_Y$ be the largest \ri~space in $W^mX_Y\hra Y$ and let $X_{M(Y)}$ be the largest \ri~space in $W^mX_{M(Y)}\hra M(Y)$.
Since $Y\hra M(Y)$, one has $W^mX_Y\hra M(Y)$, and therefore
\begin{equation} \label{E:cor-1}
	X_Y\hra X_{M(Y)}
\end{equation}
due to the optimality of $X_{M(Y)}$ in $W^mX_{M(Y)}\hra M(Y)$.
As $Y$ and $M(Y)$ are on the same fundamental level, Corollary~\ref{C:fundamental-domain-equality} asserts that so do $X_Y$ and $X_{M(Y)}$.
Hence they share the fundamental Orlicz space, \ie $L(X_Y)=L(X_{M(Y)})$.
Theorem~\ref{T:principal-alternative-sobolev-embeddings} then tells us that the non-existence of a largest Orlicz space in embedding~\eqref{E:emb-MY} is equivalent to $L(X_{M(Y)}) \not \hra X_{M(Y)}$, which, combined with~\eqref{E:cor-1}, yields
\begin{equation} \label{E:cor-2}
    L(X_{Y})\not \hra X_{Y}.
\end{equation}
Using Theorem~\ref{T:principal-alternative-sobolev-embeddings} again, relation~\eqref{E:cor-2} implies that no largest Orlicz domain space in~\eqref{E:emb-Y} exists.
\end{proof}

We shall now apply Theorem~\ref{T:nonexistence-optimal-orlicz-on-level} to show that there is no optimal Orlicz domain space $L^A(\Omega)$ in the embedding
\begin{equation}\label{E:14-new-alternative}
    W^{m} L^A(\Omega) \hra L^{\infty,\nm;-1}(\Omega),
\end{equation}
where the range space is described by the norm
\begin{equation}
	\nrm{f}_{\infty,\nm;-1}
		= \left(
				\int_{0}^{1} \biggl(\frac{f^{**}(t)}{1-\log t}\biggr)^\nm \frac{\dd t}{t}
			\right)^\mn,
\end{equation}
solving thereby the open problem mentioned in the introduction.

\begin{theorem} \label{T:bw-alternative}
Let $m,n\in\N$ be such that $n\ge2$ and $m\le n-1$, and let $\Omega\subset \R^n$ be a bounded domain with Lipschitz boundary.
Then there is no largest Orlicz space $L^A(\Omega)$ such that
\begin{equation}
    W^{m} L^A(\Omega) \hra L^{\infty,\nm;-1}(\Omega).
\end{equation}
\end{theorem}

\begin{proof}
Denote $Y=L^{\infty,\nm;-1}$.
Then, by~\citet[Lemma~3.7]{Opi:99}, its fundamental function obeys
\begin{equation*}
	\vpy(t)\approx(1-\log t)^{\mn-1}
		\quad\text{for $t\in(0,1)$.}
\end{equation*}
Furthermore, we have $M(Y)=L(Y)=\exp L^{\frac{n}{n-m}}$, the Orlicz space given by the Young function equivalent to $\exp t^{n/(n-m)}$ near infinity.
It is known that no largest Orlicz space $L^A(\Omega)$ exists in
\begin{equation}
	W^{m}L^A(\Omega) \to \exp L^\frac{n}{n-m}(\Omega),
\end{equation}
see \eg~\cite[Theorem~4.3]{Pic:98} for $m=1$ and \cite[Example~5.1b]{Mus:16} for arbitrary $m$.
The assertion then follows from Theorem~\ref{T:nonexistence-optimal-orlicz-on-level}.
\end{proof}

\begin{remark} \label{R:other-sobolev-embeddings}
The main results contained in this section are restricted to the case of Euclidean--Sobolev embedding on bounded Lipschitz domains.
We would like to point out that, following similar arguments, one can obtain analogous results for various other types of Sobolev embeddings.
Namely, the techniques developed here are immediately applicable to embeddings on less regular domains \citep{Cia:15}, first-order Gaussian--Sobolev embeddings \citep{Cia:09}, trace embeddings \citep{Cia:16}, embeddings into spaces endowed with Frostman measures \citep{Cia:20}, embeddings of fractional Orlicz--Sobolev spaces into \ri~spaces \cite{Alb:21}, and embeddings into H\"older, Campanato and Morrey spaces \citep{CiP:98,Cia:03}.
In all these cases, the characterization of the optimal \ri~domain space is contained in the quoted papers, while the non-existence of a largest Orlicz space such that the Sobolev space built upon it is embedded into an appropriate Marcinkiewicz space follows from the results of~\citet{Mus:16}.
\end{remark}

\section{Principal Alternative for Operators}\label{S:principal-alternative-for-operators}

In this section, we shall apply the idea of the principal alternative to several particular operators.
The ultimate aim is to obtain a fairly simple condition for the existence of optimal Orlicz domain or target spaces for boundedness of a given operator in the spirit of Theorem~\ref{T:principal-alternative-sobolev-embeddings}, in which we use ideas from Theorem~\ref{T:intro-principal-alternative-for-spaces}.
It will be useful to formulate first the statement for a generic quasilinear operator.

\paragraph{Quasi-linear operators}

Let $\DD\subset\MRM$ be linear.
Consider a mapping $T\colon\DD\to\measurable\SN$.
If there is a constant $C>0$ such that for all $f, g\in\DD$ and $\lambda\in\R$ one has
\begin{enumerate}
	\item $\abs{T(f+g)}\le C(\abs{Tf} + \abs{Tg})$ $\nu$-\ae in $\SS$,
	\item $\abs{T(\lambda f)} = \abs{\lambda} \abs{Tf}$,
\end{enumerate}
we shall call $T$ a \emph{quasi-linear operator} defined on $\DD$ and taking values in $\measurable(S,\nu)$.
We will also say that $T$ is a quasi-linear operator from $\DD$ to $\measurable(S,\nu)$.
\medskip

Linear operator $T$ from $\DD$ to $\measurable_0(S,\nu)$ is a special case of a quasilinear operator.
If, in such a case, $X\RM\subset\DD$ and $Y\SN$ are \ri spaces,
we shall write $T\colon X \to Y$ provided that there is some $C>0$ such that $\nrm{Tf}_{Y}\le C \nrm{f}_{X}$ for all $f\in X$.
Note that this implies that the image of $X$ under $T$ is contained in~$Y$.

\begin{theorem}[principal alternative for operators]
\label{T:principal-alternative-operators}
Let $X\RM$ and $Y\SN$ be \ri spaces and $T$ be a quasi-linear operator for which $T\colon X\RM \to Y\SN$.
\begin{enumerate}
	\item\label{en:PA-operators-domain} If $X$ is the largest \ri space for which $T\colon X \to Y$, then either $L(X)\hra X$ and $L(X)$ is the largest Orlicz space in
	\begin{equation}\label{E:principal-operator-domain}
	    T\colon L(X)\to Y,
	\end{equation}
	or no largest Orlicz domain space rendering \eqref{E:principal-operator-domain} true exists.
	\item\label{en:PA-operators-target} If $Y$ is the smallest \ri space for which $T\colon X\to Y$, then either $Y\hra L(Y)$ and $L(Y)$ is the smallest Orlicz space in
	\begin{equation}\label{E:principal-operator-target}
	    T\colon X\to L(Y),
	\end{equation}
	or no smallest Orlicz target space rendering \eqref{E:principal-operator-target} true exists.
\end{enumerate}
\end{theorem}

\begin{proof}
As $X$ is the largest \ri space for which $T\colon X \to Y$, we have, for any quasi-convex function $A$, that $T\colon L^A\to Y$ if and only if $L^A\hra X$.
Hence, $L^A$ is the largest Orlicz space in $T\colon L^A\to Y$ if and only if $L^A$ is the largest Orlicz space contained in $X$.
Therefore, the assertion follows from Theorem~\ref{T:intro-principal-alternative-for-spaces}~\ref{en:PA-spaces-domain}.
The proof of \ref{en:PA-operators-target} follows analogously
from Theorem~\ref{T:intro-principal-alternative-for-spaces}~\ref{en:PA-spaces-target}.
\end{proof}

In the rest of this section, we show the application of our principal alternative to two specific operators, namely to the Hardy-Littlewood maximal operator and the Laplace transform.

\subsection{The Hardy-Littlewood maximal operator}

Let $\Omega\subset\R^n$ be an open set.
The Hardy-Littlewood maximal operator is defined for every locally integrable function $f$ on $\Omega$ by
\begin{equation*}
	Mf(x) = \sup_{Q\ni x} \frac{1}{\abs{Q}}\int_Q \abs{f(y)}\d y
		\quad\text{for $x\in\Omega$},
\end{equation*}
where the supremum is taken over all cubes $Q$ contained in $\Omega$ with sides parallel to the coordinate axes of $\R^n$.
The operator $M$ is sub-linear, hence also quasi-linear.

Our goal is to characterize the optimal Orlicz spaces in
\begin{equation} \label{E:MLAtoLB}
	M\colon L^A(\R^n)\to L^B(\R^n).
\end{equation}
For simplicity, we will restrict ourselves to the case when $\Omega=\R^n$.
The results for arbitrary $\Omega$ with $\abs{\Omega}=\infty$ are the same, while in the case $\abs{\Omega}<\infty$, the results are analogous, but some technical modifications are needed as the behaviour of Young functions near zero is irrelevant.

We will benefit from the fact that a reduction principle for $M$ in Orlicz spaces is known.
Namely, boundedness \eqref{E:MLAtoLB} holds for two Young functions $A$ and $B$ if and only if
\begin{equation} \label{E:MineqAB}
	\int_0^t \frac{B(\tau)}{\tau^2}\d\tau
		\le \frac{A(Kt)}{t}
	\quad\text{for all $t>0$}
\end{equation}
for some $K>0$, see~\eg~\citep[Theorem~2.2]{Kit:97}.

\newcommand{\MtarA}{{B_A}}

Let us discuss the optimal target spaces first.
Suppose that $A$ is a given quasi-convex function.
Assume that
\begin{equation} \label{E:MA-condition}
	\int_0^\infty \widetilde{A}(t_0\tau)e^{-\tau}\dd\tau < \infty
		\quad\text{for some $t_0>0$}.
\end{equation}
Then let $\MtarA\colon[0,\infty]\to[0,\infty]$ denote the quasi-convex function satisfying
\begin{equation} \label{E:MBA-def}
	\widetilde{\MtarA}(t)
		= \int_0^\infty \widetilde{A}(t\tau)e^{-\tau}\dd\tau
		\quad \text{for $t\in[0,\infty]$}.
\end{equation}
If, moreover, $A$ is a Young function, then $\MtarA$ is a Young function as well.

\begin{theorem}[optimal Orlicz target space for the maximal operator]
\label{T:M-target}
Let $A$ be a Young function.
Assume that $A$ satisfies condition~\eqref{E:MA-condition} and let $\MtarA$ be the Young function from~\eqref{E:MBA-def}.
Then the following statements are equivalent.
\begin{enumerate}
	\item\label{en:Mtar-bound}
		It holds that $M\colon L^A(\R^n) \to L^{\MtarA}(\R^n)$;
	\item\label{en:Mtar-space}
		The space $L^{\MtarA}(\R^n)$ is the smallest target Orlicz space in \eqref{E:MLAtoLB};
	\item\label{en:Mtar-space-exists}
		There exists a smallest target Orlicz space in \eqref{E:MLAtoLB};
	\item\label{en:Mtar-condition}
		There exists $K>0$ such that
		\begin{equation} \label{E:MBAleA}
			\int_0^t \frac{\MtarA(\tau)}{\tau^2}\d\tau \le \frac{A(Kt)}{t}
				\quad\text{for all $t\in (0,\infty)$}.
		\end{equation}
\end{enumerate}
Conversely, if $A$ does not obey condition~\eqref{E:MA-condition}, then no Orlicz target space in~\eqref{E:MLAtoLB} exists.
\end{theorem}

\begin{proof}
Firstly, \citet[Theorem~3.1]{Edm:20} assert that the smallest \ri~space $Y$ for which $M\colon L^A\to Y$ exists if and only if
\begin{equation} \label{E:MtarY-cond}
	\log\tfrac1t \in L^{\widetilde A}(0,1)
\end{equation}
and, if so, it obeys
\begin{equation*}
	\nrm{g}_{Y'(\R^n)}
		= \nrm*{ \int_\tau^\infty \frac{g^*(r)}{r}\d r}_{L^{\widetilde{A}}(0,\infty)}
	\quad\text{for $g\in\MM(\R^n)$}.
\end{equation*}
Calculating the fundamental function of $Y'$ yields
\begin{equation} \label{E:Mtar-fund-1}
	\varphi_{Y'}(t)
		= \nrm*{ \int_\tau^\infty \chi_{[0,t)}(r)\frac{\d r}{r}}_{L^{\widetilde{A}}}
		= \nrm[\big ]{ \chi_{[0,t)}\log\tfrac t\tau}_{L^{\widetilde{A}}}
	\quad\text{for $t\ge 0$},
\end{equation}
which, by the definition of the Luxemburg norm, reads as
\begin{align} \label{E:Mtar-fund-2}
\begin{split}
	\varphi_{Y'}(t)
		& = \inf\set*{ \lambda>0: \int_{0}^{t} \widetilde{A}\brk*{\tfrac1\lambda \log\tfrac{t}{\tau}}\d\tau \le 1}
			= \inf\set*{ \lambda>0: t\int_{0}^{\infty} \widetilde{A}\brk*{\tfrac{y}{\lambda}} e^{-y}\dd y \le 1}
			\\
		& = \inf\set[\big ]{ \lambda>0: \widetilde{\MtarA}\brk*{\tfrac1\lambda} \le \tfrac1t}
			= 1/\sup\set[\big ]{ \sigma>0: \widetilde{\MtarA}(\sigma) \le \tfrac1t}
			= \brk[\big ]{\widetilde{\MtarA}}^{-1}_\#(t),
\end{split}
\end{align}
where we used the change of variables $y=\log (t/\tau)$ and the definition of the correlative function~\eqref{E:def-correlative-function}.
Now, equalities~\eqref{E:Mtar-fund-1} tell us that the general condition~\eqref{E:MtarY-cond} is equivalent to the fact that $\varphi_{Y'}(1)$ is finite, which is in our case equivalent to $A$~\eqref{E:MA-condition} due to \eqref{E:Mtar-fund-2}.
Thus, if \eqref{E:MA-condition} is not satisfied, no \ri~target space for $L^A$ exists and hence no Orlicz space can exist.

In the rest of the proof, assume that \eqref{E:MA-condition} holds.
Relation between fundamental and Young function of an Orlicz space~\eqref{E:fundamental-ri} asserts that $\varphi_{L^{\widetilde{\MtarA}}}=(\widetilde{\MtarA})^{-1}_\#$ which, together with~\eqref{E:Mtar-fund-2}, yields $L(Y')=L^{\widetilde{\MtarA}}$ or, equivalently, $L(Y)=L^{\MtarA}$.
Therefore, the equivalence of \ref{en:Mtar-bound}, \ref{en:Mtar-space} and \ref{en:Mtar-space-exists} now follows from claim~\ref{en:PA-operators-target} of Theorem~\ref{T:principal-alternative-operators}.
Finally, the equivalence of \ref{en:Mtar-bound} and \ref{en:Mtar-condition} follows by the general characterization of boundedness~\eqref{E:MLAtoLB} by condition~\eqref{E:MineqAB} mentioned in the introduction.
\end{proof}

\begin{remark}
One can show that if $\widetilde{A}$ satisfies the $\Delta_2$ condition, then $\MtarA\sim A$ and $L^\MtarA$ is always the smallest Orlicz target in~\eqref{E:MLAtoLB}, see~\eg~\citep[Theorem~1.2.1]{Kok:91}.
However, the $\Delta_2$ condition for $\widetilde{A}$ is not characterizing.

Consider for instance a Young function $A$ that satisfies $A(t)\sim t(\log\tfrac1t)^{\alpha_0}$ near zero and $A(t)\sim t(\log t)^{\alpha_\infty}$ near infinity, where $\alpha_0<-1$ and $\alpha_\infty\ge 1$.
One can infer that $\MtarA(t)\sim t(\log\tfrac1t)^{\alpha_0-1}$ near zero and $\MtarA(t)\sim t(\log t)^{\alpha_\infty-1}$ near infinity.
Then $\widetilde{A}$ does not satisfy the $\Delta_2$ condition, while $A$ and~$\MtarA$ obey condition~\eqref{E:MBAleA}.
Therefore, $M\colon L^A\to L^\MtarA$ and the target is optimal within all Orlicz spaces.
Note that $L^\MtarA$ is also optimal among all \ri~spaces as shown by~\citet[Theorem~3.3]{Edm:20}.
\end{remark}

\newcommand{\MdomB}{{A_B}}

To show the complete picture of optimal Orlicz spaces for the maximal operator, we exhibit the situation on the domain side as well.
This case is much less involved and follows directly from the reduction principle without the need of the principal alternative.

Let $B$ be a given quasi-convex function and assume that
\begin{equation} \label{E:B-cond}
	\int_{0} \frac{B(\tau)}{\tau^2}\dd\tau < \infty.
\end{equation}
We define $\MdomB\colon[0,\infty]\to [0,\infty]$ by
\begin{equation} \label{E:MAB-def}
	\MdomB(t)
		= t \int_{0}^{t} \frac{B(\tau)}{\tau^2}\dd\tau
	\quad\text{for $t\in [0,\infty]$}.
\end{equation}
Then $\MdomB$ is again a quasi-convex function (whence it is equivalent to a Young function, see equality \eqref{E:Young-from-quasiconvex} and~\eqref{E:Young-basic-equivalence}).
The characterization of optimal Orlicz domain space now reads as follows.

\begin{theorem}[optimal Orlicz domain space for the maximal operator]
\label{T:M-domain}
Let $B$ be a Young function.
If $B$ satisfies condition~\eqref{E:B-cond}, then
\begin{equation} \label{E:MLABtoLB}
	M\colon L^\MdomB(\R^n) \to L^B(\R^n),
\end{equation}
where $\MdomB$ is the function given by \eqref{E:MAB-def}.
Furthermore, the Orlicz space $L^\MdomB(\R^n)$ is the largest Orlicz domain space in~\eqref{E:MLABtoLB}.

Conversely, if condition~\eqref{E:B-cond} fails, then no Orlicz domain space rendering~\eqref{E:MLAtoLB} true exists.
\end{theorem}

\begin{proof}
Let $B$ obey condition~\eqref{E:B-cond}.
Inequality~\eqref{E:MineqAB} is clearly satisfied when $A$ is replaced by $\MdomB$ and hence, by the reduction principle, boundedness~\eqref{E:MLABtoLB} holds.
To prove the optimality, assume that $A$ obeys~\eqref{E:MLAtoLB}.
Then, due to reduction principle again, inequality~\eqref{E:MineqAB} yields $\MdomB\prec A$, whence $L^A(\R^n)\to L^{\MdomB}(\R^n)$.
The necessity of condition~\eqref{E:B-cond} follows easily from inequality~\eqref{E:MineqAB}.
\end{proof}

\begin{remark}
It is possible to show that the space $L^{\MdomB}$ is in fact the largest domain space in the class of all \ri~spaces.
This follows from the general characterization of the optimal \ri~domain space~\citep[Theorem~3.2]{Edm:20}, the fact that $(Mf)^*\approx f^{**}$ for $f\in\MM$, see~\eg~\citep[Chapter~3, Theorem~3.8]{Ben:80}, and reversed inequalities for the maximal operator~\citep[Theorem~2.6]{Kit:97}.
The details are omitted.
\end{remark}

\subsection{Laplace transform}

\newcommand{\LtarA}{{B_A}}
\newcommand{\LtarAi}{{G_A}}

The Laplace transform $\LL$ is a well-known classical linear integral operator defined for $f\in\MM(0,\infty)$ by
\begin{equation} \label{E:L-def}
	\LL f(t) = \int_{0}^\infty f(\tau)e^{-t\tau}\dd\tau
		\quad\text{for $t>0$}
\end{equation}
whenever the integral exists.
Our aim is to characterize the optimal Orlicz spaces in
\begin{equation} \label{E:LAtoB}
	\LL\colon L^A(0,\infty)\to L^B(0,\infty).
\end{equation}
We will analyze only the target side as, since the Laplace transform is self-adjoint, the domain can be obtained easily using standard duality arguments.

Let $A$ be a given Young function.
Assume that $A$ obeys
\begin{equation} \label{E:LA-cond}
	\int_{0} \log\frac{1}{A(\tau)}\dd\tau < \infty.
\end{equation}
We define functions $\LtarAi,\LtarA\colon[0,\infty]\to[0,\infty]$ by
\begin{equation} \label{E:LBA-def}
	\LtarAi(t) = t \int_{0}^{t} \frac{\widetilde{A}(\tau)}{\tau^2}\dd\tau
	\quad\text{and}\quad
	\LtarA(t)
		= \frac{1}{\LtarAi\brk{\frac{1}{t}}}
		\quad\text{for $t\ge 0$}.
\end{equation}
The characterization of the optimal Orlicz target for the Laplace transform reads as follows.

\begin{theorem} \label{T:alternative-laplace}
Let $A$ be a Young function.
If $A$ satisfies condition~\eqref{E:LA-cond}, then $\LtarA$ is a quasi-convex function and the following statements are equivalent.
\begin{enumerate}
	\item\label{en:Ltar-bound}
	It holds that $\LL\colon L^A(0,\infty)\to L^\LtarA(0,\infty)$;
	\item\label{en:Ltar-space}
	The space $L^\LtarA(0,\infty)$ is the smallest Orlicz target space in~\eqref{E:LAtoB};
	\item\label{en:Ltar-space-exists}
	There exists a smallest Orlicz target space in~\eqref{E:LAtoB}.
\end{enumerate}
Conversely, if condition~\eqref{E:LA-cond} fails, then no Orlicz target space in~\eqref{E:LAtoB} exists.
\end{theorem}

\begin{proof}
First observe that for $\varepsilon>0$, we have by Fubini's theorem that
\begin{equation}
	\int_{0}^\varepsilon \frac{a^{-1}(t)}{t}\dd t
		= \int_{0}^\varepsilon \frac{1}{t}\int_{0}^{a^{-1}(t)}\d\tau\dd t
		= \int_{0}^{a^{-1}(\varepsilon)} \int_{a(\tau)}^\varepsilon \frac{\d t}{t}\dd\tau
		= \int_{0}^{a^{-1}(\varepsilon)} \log\frac{\varepsilon}{a(\tau)}\dd\tau
\end{equation}
and therefore, using integral representation~\eqref{E:Young-conjugate-as-integral} and trivial estimates~\eqref{E:Young-trivial}, we conclude that condition~\eqref{E:LA-cond} holds if and only if the integral $\int_{0} \widetilde{A}(\tau)\tau^{-2}\dd\tau$ converges.
In this case, the function~$\LtarAi$ is finite near zero and quasi-convex.
Hence, as $\LtarA=(\LtarAi)_\#$, the function $\LtarA$ is quasi-convex as well.

Next, characterization by \citet[Theorem~3.4]{Bur:17} asserts that the optimal~\ri space $Y$ exists if and only if
\begin{equation} \label{E:LtarY-cond}
	\min\brc*{1,\tfrac1\tau} \in L^{\widetilde A}(0,\infty)
\end{equation}
and, in the affirmative case, its norm obeys
\begin{equation*}
	\nrm{g}_{Y'(0,\infty)}
		= \nrm[\bigg ]{ \int_0^{1/\tau} g^*(r)\dd r}_{L^{\widetilde{A}}(0,\infty)}
	\quad\text{for $g\in\MM(0,\infty)$}.
\end{equation*}
Testing this relation by characteristic functions together with standard computations, see~\eg~\cite[Eq.~5.11]{Cia:19} or \cite[Proposition~9.1]{Pic:98}, yields
\begin{equation} \label{E:Ltar-fund-1}
	\varphi_{Y'}(t)
		= \nrm[\big ]{ t\chi_{[0,\frac1t)}(\tau) + \tfrac1\tau\chi_{[\frac1t,\infty)}(\tau)}_{L^{\widetilde{A}}}
		\approx \nrm[\big ]{\tfrac1\tau\chi_{[\frac1t,\infty)}(\tau)}_{L^{\widetilde{A}}}
		\quad\text{for $t>0$}.
\end{equation}
By the definition of the Luxemburg norm, the change of variables and the definition of the right-continuous inverse, we have for $t\in(0,\infty)$ that
\begin{align} \label{E:Ltar-fund-2}
\begin{split}
	\varphi_{Y'}(t)
		& = \inf\set*{ \lambda>0: \int_{1/t}^{\infty} \widetilde{A}\brk*{\tfrac1{\lambda\tau}}\d\tau \le 1}
			= \inf\set*{ \lambda>0: \frac1\lambda \int_{0}^{t/\lambda} \frac{\widetilde{A}(y)}{y^2}\dd y \le 1}
			\\
		& = t \inf\set[\big ]{ \lambda>0: \LtarAi\brk*{\tfrac1\lambda} \le t}
			= t\big/\sup\set[\big ]{ \sigma>0: \LtarAi(\sigma) \le t}
			= {t}\big/{(\LtarAi)^{-1}(t)}.
\end{split}
\end{align}
From~\eqref{E:Ltar-fund-1} we observe that condition~\eqref{E:LtarY-cond} holds if and only if $\varphi_{Y'}(1)$ is finite, that is, by~\eqref{E:Ltar-fund-2}, it is equivalent to the convergence of the integral $\int_{0}\widetilde{A}(\tau)\tau^{-2}\dd\tau$ which holds if and only if~\eqref{E:LA-cond} is satisfied.
Therefore, the assertion will follow by principal alternative (Theorem~\ref{T:principal-alternative-operators}) once we show $L(Y)=L^\LtarA$.
Computations~\eqref{E:Ltar-fund-2} together with relation~\eqref{E:fundamental-associate} between fundamental function of $Y$ and $Y'$ yield
\begin{equation} \label{E:Ltar-fund-3}
	\varphi_Y\approx \brk{\LtarAi}^{-1}.
\end{equation}
Observe that $\LtarAi$ is either zero near zero, infinity near infinity, or strictly increasing.
Therefore, $\brk{\LtarAi}^{-1}$ is in fact continuous on~$(0,\infty)$.
Since $\varphi_Y$ is continuous on $(0,\infty]$ with $\varphi_Y(0)=0$, we conclude that relation~\eqref{E:Ltar-fund-3} holds for all $t\in[0,\infty]$ provided $\brk{\LtarAi}^{-1}$ denotes the left-continuous inverse of~$\LtarAi$.
Finally, by the definition of~$\LtarA$, we have
\begin{align*}
	\varphi_{Y}(t)
		&	\approx \inf\set[\big ]{\tau\ge 0: t\le \LtarAi(\tau)}
			= \inf\set[\big ]{\tau\ge 0: \LtarA(\tfrac1\tau)\le \tfrac1t}
			\\
		&	= 1\big/\sup\set[\big ]{\sigma\ge 0: \LtarA(\sigma)\le\tfrac1t}
			= \brk{\LtarA}^{-1}_\#(t)
\end{align*}
for $t\ge 0$, and therefore $L(Y)=L^\LtarA$, due to relation~\eqref{E:fundamental-ri}.
\end{proof}

\begin{remark}
It is known that $\LtarAi\sim\widetilde A$ if and only if $A$ satisfies the $\Delta_2$ condition, see~\eg~\citep[Theorem~1.2.1]{Kok:91}.
In this case, the definition of $\LtarA$ can be simplified only to $B_A=(\widetilde{A})_\#$.
\end{remark}

\begin{example}
Suppose that $A(t)=t^p$ for $p\in[1,\infty)$. Then a straightforward calculation yields $\LtarA(t)\sim t^{p'}$.
It follows from~\citep[Theorem 3.8]{Bur:17} that $\LL\colon L^p \to L^{p',p}$, in which $L^{p',p}$ is the smallest within all \ri~target spaces.
The application of principal alternative (Theorem~\ref{T:alternative-laplace}) asserts that the smallest Orlicz target space for $L^p$ exists if and only if $L^{p',p}\hra L^{p'}$, and that is the case if and only if $p\in[1,2]$.
\end{example}

The previous example in particular asserts that $\LL\colon L^2\to L^2$ where the target space is the optimal Orlicz space (it is even the optimal \ri~space).
This is the only Lebesgue space with said property.
However, in the class of Orlicz spaces, there are more examples.

\begin{example}
Let $\alpha\in\R$ and suppose that $A$ is a Young function satisfying
\begin{equation}
	A(t)\sim t^2\log^\alpha(t+\tfrac{1}{t})
		\quad\text{for $t>0$}.
\end{equation}
Then $A$ obeys the $\Delta_2$ condition and $B_A\sim A$.
Moreover, the boundedness $\LL\colon L^A \to L^A$ holds true, as proven by~\citet[Remark 3.9]{Bur:17}.
The principal alternative (Theorem~\ref{T:alternative-laplace}) thus dictates that $L^A$ is the smallest Orlicz target space in $\LL\colon L^A\to L^A$.
Note that $L^A$ is in fact also the largest Orlicz domain space therein, since $\LL$ is self-adjoint.
\end{example}

Unlike for the Hardy-Littlewood maximal operator, the characterizing closed-form condition involving $A$ and $B$ in the spirit of inequality~\eqref{E:MineqAB} for embedding~\eqref{E:LAtoB} to hold is missing.
Therefore, we close this section with a rather general condition for the existence of the smallest Orlicz target space in~\eqref{E:LAtoB} which is sufficient but not necessary.
It follows directly from the principal alternative (Theorem \ref{T:alternative-laplace}) and a classical interpolation result for operators bounded between $L^1 \to L^\infty$ and $L^2 \to L^2$ (such as the Fourier, or the Laplace transforms), see~\citep[Theorem~3.10]{Jod:70}.

\begin{corollary}
Suppose $A$ is a Young function satisfying the $\Delta_2$ condition and such that $t\mapsto {A(t)}/{t^2}$ is decreasing.
Then
\begin{equation} \label{E:Laplace-decr-bdd}
	\LL\colon L^A(0,\infty) \to L^{(\widetilde{A})_\#}(0,\infty),
\end{equation}
and the space $L^{(\widetilde{A})_\#}$ is the smallest Orlicz target space rendering~\eqref{E:Laplace-decr-bdd} true.
\end{corollary}

\appendix

\section{Sub-diagonal and Uniformly Sub-diagonal R.I.~Spaces} \label{S:unioins-of-orlicz-spaces}

In this appendix, we study the unions and intersections of Orlicz spaces in more detail.
The material contained here, while naturally connected with Section \ref{S:unions}, goes in a different direction than the rest of the paper, which is why it is collected in an appendix.
To simplify our expository, we introduce two notions explaining the appendix's title.
\begin{definition}
Let $X$ be an \ri space. We say that $X$ is \emph{sub-diagonal}, if
\begin{equation}
    X=\bigcup \set[\big ]{L^A: \text{$A$ is a Young function such that $L^A\hra X$}},
\end{equation}
and we say that $X$ is \emph{uniformly sub-diagonal} if
\begin{equation}
    X=\bigcup \set[\big ]{L^A: \text{$A$ is a Young function such that $L^A\hrastar X$}},
\end{equation}
where ``$\hrastar$'' denotes the absolutely continuous, also called almost-compact, embedding defined below.
\end{definition}

Using this terminology, the final result of Section~\ref{S:unions} (Theorem~\ref{T:lambda-union}), asserts that every single Lorentz space $\Lambda^E$ is sub-diagonal.
Recalling the Luxemburg representation theorem (see Section~\ref{S:principal-alternative}), this holds when the spaces are considered over an arbitrary $\sigma$-finite non-atomic measure space.
We shall further extend this result and study when a Lorentz space $\Lambda^E$ is in fact uniformly sub-diagonal (Corollary~\ref{C:unions-AC-Lorentz}).
Next, we extend these results to cases when the Lorentz space is replaced with a more general \ri~space.
We provide both abstract results (Theorem \ref{T:lifting-principle}) and concrete results for well-known scales of spaces such as the Lorentz $L^{p,q}$ spaces (Corollary \ref{C:uniform-diagonality-of-Lorentz} and, more generally, Theorem \ref{T:telefonni-seznam}).
In fact, this scale will play the role of a ``model example'' for us, uncovering the origin of the terminology ``sub-diagonal'', as we will be able to show that $L^{p,q}$ is (uniformly) sub-diagonal if and only if $q\leq p$.

Moreover, to tackle the problems related to the almost-compact embeddings, we shall, en passant, characterize the embedding $L^A \hrastar \Lambda^E$ (Proposition~\ref{P:LA-LambdaE-AC}) as well as provide a sufficient condition on $E$ under which one has
\begin{equation*}
	L^A\hra \Lambda^E
		\quad\text{if and only if}\quad
	L^A\hrastar \Lambda^E,
\end{equation*}
for every Young function $A$ (Theorem \ref{T:unions-AC-Lorentz}).
All of these results are new.

We begin with recalling some definitions and known results.

\paragraph{Lorentz $L^{p,q}$ spaces}
Let $\RM$ be a $\sigma$-finite non-atomic measure space.
Let $p,q\in(0,\infty]$ and define, for any $f\in\measurable\RM$, the \emph{Lorentz functional}
\begin{equation}
	\nrm{f}_{p,q}
		= \begin{cases}\displaystyle
				\left(
					\int_{0}^{\infty} \bigl[t^\ip f^{*}(t)\bigr]^q \frac{\dd t}{t}
				\right)^\iq
					& \text{if $q\in(0,\infty)$},
					\\[\bigskipamount]
				\sup\limits_{t\in(0,\infty)} t^\ip f^{*}(t)
					& \text{if $q=\infty$},
			\end{cases}
\end{equation}
and the respective \emph{Lorentz} space $L^{p,q}=\set{f\in\MRM: \nrm{f}_{L^{p,q}}<\infty}$.
If $p\in[1,\infty)$ and $1\le q\le p$, the space $L^{p,q}$ is an \ri space.
If $p\in(0,\infty]$, then $L^{p,p}=L^p$.
If $p=\infty$ and $q<\infty$, then $L^{p,q}$ contains only the zero function.
Finally, if $p\in (1,\infty)$ and $q\in[1,\infty]$, then $L^{p,q}$ can be equivalently renormed as to become an \ri space.
Moreover, with a fixed $p\in(1,\infty]$, one has $L^{p,q}\subset L^{p,r}$ if and only if $q\le r$, irrespectively of the underlying measure space.
Setting $A(t)=t^p$ with $p\in(1,\infty)$, one has
\begin{equation} \label{E:lorentz-relations}
	L^A = L^{p,p},
		\quad
	\Lambda^A = L^{p,1},
		\quad
	M^A = L^{p,\infty},
\end{equation}
with equivalent norms.

\paragraph{Absolutely continuous norms}

Let $\RM$ be a $\sigma$-finite non-atomic measure space.
We say that a sequence $\set{E_n}$ of measurable sets (not necessarily of finite measures) \emph{converges \ae to $\emptyset$}, denoted as $E_n\to\emptyset$ \ae, if the characteristic functions $\chi_{E_n}$ converge to the zero function pointwise \ae.

Let $X$ be an $\ri$ space over $\RM$.
A function $f\in X$ is said to have \emph{absolutely continuous norm} in $X$ if $\nrm{f\chi_{E_n}}_X\to 0$ for every sequence $\set{E_n}$ satisfying
$E_n\to\emptyset$ \ae.
The set of all functions in $X$ of an absolutely continuous norm is denoted by $X_a$.
If $X=X_a$, then we say that $X$ \emph{has absolutely continuous norm}.

For Lebesgue spaces, we have $L^p=L^{p}_a$ if and only if $p<\infty$, and $L^\infty_a=\set{0}$.
For Lorentz spaces, we have $L^{p,q}=L^{p,q}_a$ if and only if $q<\infty$.

For Orlicz spaces, we will need to define two complementary conditions.
We say that a quasi-convex function $A$ satisfies the \emph{$\Delta_2$ condition near infinity/globally}, if there is some $C>0$ such that $A(2t)\leq CA(t)$ near infinity/globally.
We also write $A\in\Delta_2$.
Also, $A$ satisfies the \emph{$\nabla_2$ condition near infinity/globally}, if there is $c>0$ such that $2cA(t)\le A(ct)$ near infinity/globally. We also write $A\in\nabla_2$.
If $A$ is a quasi-convex function, then
$L^A_a=\set{f\in\measurable: \varrho_A(f/\lambda)<\infty\;\text{for all $\lambda>0$}}$
and $L^{A}=L^{A}_a$ if and only if $A$ satisfies the $\Delta_2$ condition (near $\infty$ if $\MR<\infty$ and globally if $\MR=\infty$).

\paragraph{Almost compact embeddings}

Let $X$ and $Y$ be $\ri$ spaces over the same $\sigma$-finite nonatomic measure space $\RM$.
We say that $X$ is \emph{almost compactly embedded} into $Y$, written $X\hrastar Y$, if for every sequence $\set{E_n}$ satisfying
$E_n\to\emptyset$ \ae one has
\begin{equation}
	\lim_{n\to\infty} \sup_{\nrm{f}_X\le1} \nrm{f\chi_{E_n}}_Y = 0.
\end{equation}

Note that $X\hrastar Y$ if and only if $\widebar{X}\hrastar \widebar{Y}$, where $\widebar{X}$ and $\widebar{Y}$ are Luxemburg representation spaces of $X$ and $Y$, respectively.
If $\mu(\RR)<\infty$, then the almost compact embedding $X\hrastar Y$ can be thus characterised by
\begin{equation} \label{E:AC-characterisation}
	\lim_{t\to 0_+} \sup_{\nrm{f}_{\widebar X}\le1} \nrm{f^*\chi_{(0,t)}}_{\widebar Y} = 0.
\end{equation}

Next, $X\hrastar Y$ implies $X\subset Y_a$ and, moreover,
the embedding $X\hrastar Y$ is possible only if $\MR<\infty$, see~\citep[Theorem~4.5]{Sla:12}. This, together with the characterisation of almost-compact embeddings via \eqref{E:AC-characterisation} allows us to restrict our attention to spaces over $(0,1)$ without loss of generality.
Recall an important characterization by~\citet[Theorem 3.1]{Sla:12} asserting that $X\hrastar Y$ if and only if
\begin{equation} \label{E:Lenka-char}
	\text{for every sequence $\set{f_n}$ bounded in $X$ satisfying $f_n\to 0$ \ae, one has $\nrm{f_n}_Y\to 0$}.
\end{equation}

\medskip

To study which spaces $\Lambda^E$ are uniformly sub-diagonal,
we will focus on the question for which quasi-convex functions $A$ and $E$ one has $L^A\hrastar\Lambda^E$.
Let us first recall a few definitions from Section~\ref{S:unions}.
For a given quasi-convex function $E$, the Young function $G$ is given by
\begin{equation} \label{E:G-def-2}
	G(t) = \int_{0}^{t} g(\tau)\dd\tau
		\quad\text{for $t\in[0,\infty]$},
\end{equation}
where
\begin{equation} \label{E:g-def-2}
	g(\tau)
		= \frac{E_\#(\tau)}{\tau}
		= \frac{1}{\tau E\bigl(\tfrac1\tau\bigr)}
	\quad\text{for $\tau\in(0,\infty)$}
\end{equation}
and the function $w\colon(0,\infty)\to[0,\infty]$ is defined as
\begin{equation} \label{E:w-def-2}
	w(\tau) = \frac{1}{g\bigl(G^{-1}(\tau)\bigr)}
		\quad\text{for $\tau\in(0,\infty)$}.
\end{equation}
Some key properties of $G$ and $w$ are summarised in Lemmas~\ref{L:LambdaE-G} and \ref{L:w}.

\begin{lemma} \label{L:w-in-Orlicz-equivalence}
let $B$ be a Young function and $E$ a quasi-convex functions.
Suppose that $w$ is the function associated with $E$ as in~\eqref{E:w-def-2}.
Then
\begin{equation} \label{E:w-in-Orlicz-i}
	w\in L^B(0,1)
		\quad\text{or}\quad
	w\in L^B_a(0,1)
\end{equation}
if and only if
\begin{equation} \label{E:w-in-Orlicz-ii}
	\int_{1/\tau_\infty} b\Bigl(\lambda sE\bigl(\tfrac{1}{s}\bigr)\Bigr)\dd s < \infty
\end{equation}
holds for some or every $\lambda>0$, respectively, where $\tau_\infty=\sup\set{\tau\ge 0: E(\tau)<\infty}$.
\end{lemma}

\begin{proof}
By definition of the spaces $L^B$ and $L^B_a$, statement~\eqref{E:w-in-Orlicz-i}
holds if and only if
\begin{equation}
	\int_{0}^1 B\bigl(\sigma w(r)\bigr)\dd r < \infty
\end{equation}
for some or every $\sigma>0$, respectively.
Note that the convergence is relevant only at zero.
By definition of $w$, we have
\begin{equation} \label{E:w-in-Orlicz-eq}
	\int_{0} B\bigl(\sigma w(\tau)\bigr)\dd\tau
		= \int_{0} B\left( \frac{\sigma}{g\bigl(G^{-1}(\tau)\bigr)} \right)\d\tau
		= \int_{1/\tau_\infty} B\left( \frac{\sigma}{g(s)} \right)g(s)\dd s,
\end{equation}
where we changed variables $\tau=G(s)$ as $G$ is increasing, absolutely continuous and $G'=g$ \ae in a right neighbouring of $1/\tau_\infty$ mapped onto right neighbouring of zero.
Note that the equality in \eqref{E:w-in-Orlicz-eq} means that the integral on the left-hand side converges if and only if the integral on the right-hand side converges.
Finally, convergence of the last integral in~\eqref{E:w-in-Orlicz-eq} for some or every $\sigma>0$ is equivalent~to
\begin{equation} \label{E:w-in-Orlicz-eq-2}
	\int_{1/\tau_\infty} b\left( \frac{\lambda}{g(s)} \right)\d s < \infty
\end{equation}
for some or every $\lambda>0$ due to trivial inequalities~\eqref{E:Young-trivial}.
By definition of $g$, \eqref{E:w-in-Orlicz-eq-2} matches \eqref{E:w-in-Orlicz-ii}.
\end{proof}

\begin{lemma} \label{L:LA-LambdaE-AC}
Let $A$ be a Young function and $E$ be a quasi-convex function.
Let $w$ be the function associated to $E$ as in~\eqref{E:w-def-2}.
Then $L^A(0,1)\hrastar \Lambda^E(0,1)$ if and only if $E$ is finite-valued and $w\in L^{\widetilde A}_a(0,1)$.
\end{lemma}

\begin{proof}
Using characterization~\eqref{E:AC-characterisation}, the embedding $L^A\hrastar\Lambda^E$ is equivalent to
\begin{equation} \label{E:LA-LambdaE-AC-char}
	\lim_{t\to 0_+} \sup\set{\nrm{\chi_{(0,t)}f^*}_{\Lambda^E} : \nrm{f}_{L^A}\le 1} = 0.
\end{equation}
Since, by Lemmas~\ref{L:LambdaE-G} and \ref{L:w},
\begin{equation*}
	\nrm{\chi_{(0,t)}f^*}_{\Lambda^E}	
		\le t_0\nrm{f}_\infty + \int_{0}^{t} f^*w
		\le 2\nrm{\chi_{(0,t)}f^*}_{\Lambda^E}	
\end{equation*}
with $t_0=1/\sup\set{\tau\ge0:E(\tau)<\infty}$, condition~\eqref{E:LA-LambdaE-AC-char} holds if and only if $E$ is finite-valued and
\begin{equation*}
	\lim_{t\to 0_+} \sup\set*{\int_{0}^{t}f^*w : \nrm{f}_{L^A}\le 1} = 0
\end{equation*}
which, by duality, reads as $\lim_{t\to 0_+}\nrm{\chi_{(0,t)}w}_{L^{\widetilde A}}=0$. This is equivalent to $w\in L^{\widetilde A}_a$ and we are done.
\end{proof}

\begin{proposition}
\label{P:LA-LambdaE-AC}
Let $A$ be a Young function and $E$ a quasi-convex function.
Then $L^A(0,1)\hrastar \Lambda^E(0,1)$ if and only if
\begin{equation} \label{E:LA-LambdaE-AC-condition}
	\int_{0} a^{-1}\Bigl(\lambda sE\bigl(\tfrac{1}{s}\bigr)\Bigr)\dd s < \infty
	\quad\text{for every $\lambda>0$}.
\end{equation}
\end{proposition}

\begin{proof}
By Lemma~\ref{L:LA-LambdaE-AC}, $L^A\hrastar \Lambda^E$ if and only if $E$ is finite-valued and $w\in L^{\widetilde A}_a$, where $w$ is the function associated with $E$ as in~\eqref{E:w-def-2}.
Now, if condition~\eqref{E:LA-LambdaE-AC-condition} holds, then $E$ has to be finite-valued and $w$ has absolutely continuous norm in in $L^{\widetilde A}$ due to Lemma~\ref{L:w-in-Orlicz-equivalence}.
Conversely, $w\in L^{\widetilde A}_a$ and $E$ being finite-valued imply condition~\eqref{E:LA-LambdaE-AC-condition} again by Lemma~\ref{L:w-in-Orlicz-equivalence}.
\end{proof}

\begin{theorem} \label{T:unions-AC-Lorentz}
Let $A$ be Young and $E$ a finite-valued quasi-convex function.
If $E\in\nabla_2$ near infinity, then $L^A(0,1)\hra \Lambda^E(0,1)$ if and only if $L^A(0,1)\hrastar \Lambda^E(0,1)$.
\end{theorem}

\begin{proof}
Assume that $L^A\hra\Lambda^E$, \ie
	$\sup\set*{\nrm{f}_{\Lambda^E} : \nrm{f}_{L^A}\le 1} < \infty$.
Since $E$ is finite-valued, we have by Lemmas~\ref{L:LambdaE-G} and~\ref{L:w} that
	$\sup\set{\int_{0}^1 fw : \nrm{f}_{L^A}\le 1} < \infty$
which by duality means that $w\in L^{\widetilde A}$.
Using Lemma~\ref{L:w-in-Orlicz-equivalence}, the last is equivalent to
\begin{equation} \label{E:w-lambda-some}
	\int_{0} a^{-1}\Bigl(\lambda sE\bigl(\tfrac{1}{s}\bigr)\Bigr)\dd s < \infty
\end{equation}
for some $\lambda>0$.
Our goal is to show that condition~\eqref{E:w-lambda-some} holds in fact for every $\lambda>0$ as Lemma~\ref{L:w-in-Orlicz-equivalence} would then imply that $w\in L^{\widetilde A}_a$ which, in turn, yields the embedding $L^A\hrastar\Lambda^E$, due to Lemma~\ref{L:LA-LambdaE-AC}.

Since $E$ obeys $\nabla_2$ near infinity, there is $c>0$ such that $2cE(t)\le E(ct)$ near infinity, or equivalently,
\begin{equation*}
	2sE\bigl(\tfrac1s\bigr)
		\le \tfrac sc E\bigl(\tfrac cs\bigr)
	\quad\text{for $s$ near zero.}
\end{equation*}
Therefore
\begin{equation*}
	\int_{0} a^{-1}\Bigl(2\lambda sE\bigl(\tfrac{1}{s}\bigr)\Bigr)\dd s
		\le \int_{0} a^{-1}\Bigl(\lambda \tfrac sc E\bigl(\tfrac{c}{s}\bigr)\Bigr)\dd s
		= \frac1c \int_{0} a^{-1}\Bigl(\lambda sE\bigl(\tfrac{1}{s}\bigr)\Bigr)\dd s
		< \infty
\end{equation*}
and the validity of~\eqref{E:w-lambda-some} for $\lambda$ implies the validity also for $2\lambda$.
By induction and monotonicity, condition~\eqref{E:w-lambda-some} therefore holds for all $\lambda>0$.
\end{proof}

The next result is a direct consequence of Theorems~\ref{T:lambda-union} and~\ref{T:unions-AC-Lorentz}.

\begin{corollary} \label{C:unions-AC-Lorentz}
Let $E$ be a finite-valued quasi-convex function.
If $E\in\nabla_2$ near infinity, then $\Lambda^E$ is uniformly sub-diagonal.
\end{corollary}

\begin{example} \label{EX:Lp1}
Let $E(t)=t^p$ for $p\in[1,\infty)$.
Then $E$ satisfies $\nabla_2$ near infinity if and only if $p>1$.
Therefore, by Corollary~\ref{C:unions-AC-Lorentz} and relations \eqref{E:lorentz-relations}, we conclude that $L^{p,1}$ is uniformly sub-diagonal for $p>1$.
However, if $p=1$, then $L^{1,1}=L^1$ is also uniformly sub-diagonal as it follows directly from the Theorem of de la Vall\'e Poussin \citep[Section~1.2, Theorem~2]{Rao:91}, or from Theorem~\ref{T:unions-AC-Orlicz} below.
\end{example}

\begin{remark}
It follows from Example~\ref{EX:Lp1} that the condition of $E\in\nabla_2$ near infinity is not characterising in Corollary~\ref{C:unions-AC-Lorentz}.
\end{remark}

Our next goal will be to study when an analogy of Corollary~\ref{C:unions-AC-Lorentz} holds upon replacing $\Lambda^E$ with a different space.
First, for Orlicz spaces, the situation is quite easy.

\begin{theorem} \label{T:unions-AC-Orlicz}
Let $E$ be a quasi-convex function. Then $L^E$ is uniformly sub-diagonal
if and only if $E\in\Delta_2$ near infinity.
\end{theorem}

\begin{proof}
Assume that $E\in\Delta_2$ near infinity and let $f\in L^E=L^E_a$.
Then the composition $E\circ f$ belongs to $L^1$.
By the de la Vall\'ee--Poussin theorem \cite[see~\eg][Section 1.2, Theorem 2]{Rao:91}, there is a Young function $F$ satisfying $F(t)/t\to\infty$ as $t\to\infty$ and such that $E\circ f\in L^{F}$.
If we denote $A=F\circ E$, it is $f\in L^A$.
Next, we have for any $K>0$ that
\begin{equation*}
	\frac{A(t)}{E(Kt)}
		= \frac{F\brk[\big ]{E(t)}}{E(t)}\cdot\frac{E(t)}{E(Kt)}
	\quad\text{for $t\in(0,\infty)$}.
\end{equation*}
Since $E(t)\to\infty$, it follows that ${F(E(t))}/{E(t)}\to\infty$ as $t\to\infty$.
On the other hand, $E\in\Delta_2$ near infinity whence ${E(t)}/{E(Kt)}$ is uniformly bounded away from zero for large values of $t$.
Therefore ${A(t)}/{E(Kt)}\to\infty$ as $t\to\infty$ for every $K>0$, which is equivalent to $L^A\hrastar L^E$, see~\eg~\cite[Theorems~4.17.7 and~4.17.9]{FS}.
Consequently, we have
\begin{equation}\label{E:unif-subdiag-Orlicz}
	L^E = \set*{L^A: \text{$A$ is a Young function such that $L^A\hrastar L^E$}}.
\end{equation}
Conversely, assume that $E\notin\Delta_2$ near infinity.
Then there is $f\in L^E \setminus L^E_a$, whence for every $L^A$ with $L^A\hrastar L^E$, we have $f\in L^E\setminus L^A$ and therefore relation \eqref{E:unif-subdiag-Orlicz} cannot hold.
\end{proof}

Note that the sufficient condition in Corollary~\ref{C:unions-AC-Lorentz} and the characterising condition in Proposition~\ref{T:unions-AC-Orlicz} are in a sense complementary to each other.
In fact, one has $E\in\Delta_2$ near infinity if and only if $\widetilde{E}\in\nabla_2$ near infinity.
The difference in the behaviour is illustrated in the following example.

\begin{example}\label{EX:exp-spaces}
Consider a Young function satisfying $E(t)=e^t$ near infinity.
Then it is easily seen that $E\in\nabla_2$ near infinity but $E\not\in\Delta_2$ near infinity.
Hence, using Corollary~\ref{C:unions-AC-Lorentz} and Theorem~\ref{T:unions-AC-Orlicz}, we have that $\Lambda^E$ is uniformly sub-diagonal while $L^E$ is not.
Observe that the spaces $L^E$ and $\Lambda^E$ share the same fundamental function, yet the behaviour described above differs significantly.
\end{example}

Our next aim is to present a certain kind of lifting construction which will enable us to effectively transfer (uniform) sub-diagonality through scales of spaces.
For example, it can be immediately applied to characterise (uniform) sub-diagonality of Lorentz $L^{p,q}$ spaces.

\paragraph{Lifting principle}
Let $X$ be an \ri space and let $F$ be a Young function. We then define the functional $\nrm{\cdot}_{F(X)}\colon\measurable\to[0,\infty]$ by
\begin{equation}
	\nrm{f}_{F(X)} = \inf\set*{\lambda>0:\nrm*{F\brk[\big ]{{\abs{f}}/{\lambda}}}_X\le 1}
		\quad\text{for $f\in\measurable$.}
\end{equation}
We shall denote by $F(X)$ the collection of all $f\in\measurable$ such that $\nrm{f}_{F(X)}<\infty$.
Note that the functional $f\mapsto\nrm{F(\abs{f})}_{X}$ is a \emph{semimodular} in a sense by~\citet[Definition~2.1.1]{Die:11}.
Therefore, by~\citet[Theorem~2.1.7]{Die:11}, the functional $\nrm{\cdot}_{F(X)}$ is a norm.
Moreover, it can be easily observed that it is also rearrangement invariant.
If $F(t)=t^p$ for some $p\in[1,\infty)$, we simply write $X^p$ instead of $F(X)$.
Important examples are the lifting of a Lorentz and an Orlicz space.
If $p\in(1,\infty)$, $q\in[1,\infty]$ and $r\in[1,\infty)$, then
\begin{equation}
	(L^{p,q})^r = L^{pr,qr}.
\end{equation}
Next, if $A$ and $F$ are Young functions, then
\begin{equation} \label{E:lifting-Orlicz-example}
	F(L^A) = L^{A\circ F}.
\end{equation}

The space $F(X)$ has been first introduced by \citet[p.~122]{Cal:64}.
It is also a special case of a more general, so-called \emph{Calder\'on-Lozanovski\v{\i} construction}, \cf~\citep{Lozanovski2}.
For more recent results, one may consult~\citet{KLM}.

For convergence in the space $F(X)$, we recall that a sequence of measurable functions $f_n\in F(X)$ satisfies $\nrm{f_n}_{F(X)}\to 0$ if and only if $\nrm{F(K f_n)}_{X}\to 0$ for every $K>0$.
Therefore, it can be easily seen that if $F\in\Delta_2$ near infinity and the underlying measure space is finite, one has
\begin{equation}\label{E:modular-convergence-equivalence}
	\nrm{f_n}_{F(X)}\to 0
		\quad\text{if and only if}\quad
	\nrm{F(f_n)}_{X} \to 0.
\end{equation}

In the next lemma, we show that the lifting operation is in some sense (uniformly) monotone.

\begin{lemma} \label{L:lifting-hrastar}
Let $X$ and $Y$ be two \ri spaces and $F$ be a Young function.
Then
\begin{equation}
	X\hra Y
		\quad\text{implies}\quad
	F(X)\hra F(Y),
\end{equation}
and, if $F$ moreover satisfies the $\Delta_2$ condition near infinity, then
\begin{equation}
	X\hrastar Y
		\quad\text{implies}\quad
	F(X)\hrastar F(Y).
\end{equation}
\end{lemma}

\begin{proof}
The first part follows directly from the definition.
To show the second one, let $\set{f_n}$ be a sequence bounded in $F(X)$ such that $f_n\to0$ \ae.
Using characterization~\eqref{E:Lenka-char}, we need to show that $\nrm{f_n}_{F(Y)}\to 0$.
Without loss of generality, we may assume that $\nrm{F(f_n)}_X\le 1$ for every $n$.
In other words, the sequence $\set{F(f_n)}$ is bounded in $X$.
Moreover, since $F(t)\to0$ as $t\to 0^+$, also $F(f_n)\to 0$ \ae.
Therefore, using characterization~\eqref{E:Lenka-char} of $X\hrastar Y$, we get that $\nrm{F(f_n)}_Y\to 0$.
Finally, since $F$ satisfies the $\Delta_2$ condition, we employ characterization~\eqref{E:modular-convergence-equivalence} to infer that $\nrm{f_n}_{F(Y)}\to 0$ as desired.
\end{proof}

\begin{theorem}[lifting principle]\label{T:lifting-principle}
Let $X(0,1)$ be an \ri space and let $F$ be a Young function.
If $X(0,1)$ is sub-diagonal, then so is $F(X)$.
If $X(0,1)$ is uniformly sub-diagonal and $F$ satisfies the $\Delta_2$ condition near infinity, then $F(X)$ is uniformly sub-diagonal.
\end{theorem}

\begin{proof}
Assume first that $X$ is sub-diagonal and $f\in F(X)$.
We need to find a Young function $B$ such that $f\in L^B\hra F(X)$.
There exists $\lambda>0$ such that $F({\abs{f}}/{\lambda})\in X$.
Since $X$ is sub-diagonal, there is a Young function $A$ such that $F({\abs{f}}/{\lambda})\in L^A\hra X$.
Therefore, for some $\eta\ge 1$, one has
\begin{equation*}
	\int_{\RR} A\brk[\bigg ]{\frac{1}{\eta} F\brk[\bigg ]{\frac{\abs{f}}{\lambda}}} \dd\mu \le 1.
\end{equation*}
Now, set $B=A\circ F$. Then $B$ is a Young function and $f\in L^B$ since
\begin{equation*}
	\int_{\RR} B\brk[\bigg ]{\frac{\abs{f}}{\eta\lambda}} \dd\mu
		= \int_{\RR} A\brk[\bigg ]{F\brk[\bigg ]{\frac{\abs{f}}{\eta\lambda}}} \dd\mu
		\le \int_{\RR} A\brk[\bigg ]{\frac{1}{\eta} F\brk[\bigg ]{\frac{\abs{f}}{\lambda}}} \dd\mu
		\le 1.
\end{equation*}
It suffices to show that $L^B\hra F(X)$.
Since $L^A\hra X$, we have $F(L^A)\hra F(X)$, and as $F(L^A)=L^B$ due to example~\eqref{E:lifting-Orlicz-example}, the assertion follows.

Now assume that $X$ is uniformly sub-diagonal and $F$ satisfies the $\Delta_2$ condition.
Analogously to the previous part, to a given $f\in F(X)$, there is a Young function $A$ such that $L^A\hrastar X$ and $f\in L^{A\circ F}$.
Now, using Lemma~\ref{L:lifting-hrastar} with example~\eqref{E:lifting-Orlicz-example}, it is $L^{A\circ F}=F(L^A)\hrastar F(X)$.
As $A\circ F$ is a Young function, $F(X)$ is uniformly sub-diagonal.
\end{proof}

\begin{corollary}\label{C:uniform-diagonality-of-Lorentz}
If $p\in (1,\infty)$ and $q\in[1,\infty]$ then the space $L^{p,q}(0,1)$ is sub-diagonal if and only if it is uniformly sub-diagonal, which is the case if and only if $q\le p$.
\end{corollary}

\begin{proof}
If $q\le p$, then the space $L^{\frac pq,1}$ is uniformly sub-diagonal by Example~\ref{EX:Lp1}.
Therefore $\brk[\big ]{L^{\frac{p}{q},1}}^q=L^{p,q}$ is uniformly sub-diagonal thanks to the lifting principle (Theorem~\ref{T:lifting-principle}).
Conversely, when $q>p$, it is immediate that $\bigcup\set{L^E\colon L^E\hra L^{p,q}}=L^{p,p}$ which is strictly smaller than $L^{p,q}$.
\end{proof}

We remark here that in the previous corollary and theorem one may consider the spaces over the entire $(0,\infty)$ for the results not-concerning almost-compact embeddings.

We conclude this section by listing, in the form of a theorem, various diagonality properties of particular examples of rearrangement-invariant Banach function spaces, which we have shown in the previous text.
Recall that the space $\Lambda^q_w$ is defined for $q<\infty$ as in \eqref{E:classical-lorentz}.

\begin{theorem}\label{T:telefonni-seznam}
Suppose all of the spaces in what follows are over $(0,1)$.
Let $E$ be a Young function, $p\in(1,\infty)$, $q\in[1,\infty]$, $w$ a non-increasing, positive function on $[0,1]$.
Then we have the following.
\begin{enumerate}
	\item\label{en:ts-lebesgue} The space $L^q$ is sub-diagonal. It is uniformly sub-diagonal if and only if $q<\infty$.
	\item\label{en:ts-orlicz} The space $L^E$ is sub-diagonal. It is uniformly sub-diagonal if and only if $E\in\Delta_2$ near infinity.
	\item\label{en:ts-lorentz} The space $L^{p,q}$ is sub-diagonal if and only if it is uniformly sub-diagonal that is if and only if~$q\le p$.
	\item\label{en:ts-lambdaE} The space $\Lambda^E$ is sub-diagonal. It is uniformly sub-diagonal if $E\in\nabla_2$ near infinity.
	\item\label{en:ts-lambdaqw} The space $\Lambda^q_w$ is sub-diagonal. It is uniformly sub-diagonal if $2w(ct)\leq w(t)$ for some $c>0$.
\end{enumerate}
\end{theorem}
\begin{proof}
Statement~\ref{en:ts-orlicz} is exactly Theorem~\ref{T:unions-AC-Orlicz} and \ref{en:ts-lebesgue} is a special case of \ref{en:ts-orlicz}.
Claim~\ref{en:ts-lorentz} is Corollary~\ref{C:uniform-diagonality-of-Lorentz}.
Observe that the condition on $w$ implies that $\Lambda^1_w=\Lambda^E$ for a finite-valued concave function $E$ having $E\in\nabla_2$ near infinity.
Therefore statement~\ref{en:ts-lambdaqw} in case $q=1$ and statement~\ref{en:ts-lambdaE} are merely equivalent ways of stating Theorem~\ref{T:lambda-union} and Corollary \ref{C:unions-AC-Lorentz}.
What remains is the general case of \ref{en:ts-lambdaqw}.
By definition of the classical Lorentz space, we have $\Lambda^q_w=(\Lambda^1_w)^q$ and the assertion follows immediately from the lifting principle (Theorem \ref{T:lifting-principle}).
\end{proof}

\begin{remark} \label{R:super-diagonal}
We may also define dual notions of sub-diagonality.
For an \ri space $X$, we may say that $X$ is \emph{super-diagonal} or \emph{uniformly super-diagonal} if
\begin{equation*}
	X = \bigcap\set*{L^A: X\hra L^A},
		\quad\text{or}\quad
	X = \bigcap\set*{L^A: X\hrastar L^A},
\end{equation*}
respectively.
Using duality in \ri~spaces, we infer that $X$ is (uniformly) super-diagonal if and only if $X'$ is (uniformly) sub-diagonal, see also the proof of Theorem~\ref{T:marcinkiewicz-intersection}.
This allows one to immediately generate results dual to those we have provided in this appendix.
The details are omitted.
\end{remark}

\section*{Acknowledgement}

This work was supported by:
\begin{itemize}
\item Danube Region Grant no.~8X2043;
\item Czech Science Foundation, grant no.~P201/21-01976S;
\item Operational Programme Research, Development and Education, Project Postdoc2MUNI no.\ CZ.02.2.69/0.0/0.0/18\_053/0016952;
\item Grant Agency of the Charles University, project no. 327321.
\end{itemize}
Part of the work on this project was carried out during the meeting \emph{Per Partes} held at Horn\'{\i} Lyse\v ciny, June 2-6, 2021.

\bibliographystyle{abbrvnat}

\end{document}